\documentclass[11pt]{amsart}

\usepackage{latexsym}
\usepackage{amssymb}
\usepackage{amsmath}
\usepackage{color}
\usepackage{bbm}

\usepackage{graphicx}
\usepackage{tcolorbox}
\usepackage{tabularx}

\usepackage{tikz}
\usepackage{xparse}

\usepackage[a4paper]{geometry}
\NewDocumentCommand\DownArrow{O{2.0ex} O{black}}{%
   \mathrel{\tikz[baseline] \draw [<-, line width=0.5pt, #2] (0,0) -- ++(0,#1);}
}
\NewDocumentCommand\UpArrow{O{2.0ex} O{black}}{%
   \mathrel{\tikz[baseline] \draw [->, line width=0.5pt, #2] (0,0) -- ++(0,#1);}
}

\newtheorem{theorem}{Theorem}[section]
\newtheorem{lemma}[theorem]{Lemma}
\newtheorem{proposition}[theorem]{Proposition}
\newtheorem{corollary}[theorem]{Corollary}
\newtheorem{definition}[theorem]{Definition}

\newtheorem{remark}[theorem]{Remark}

\newcommand{\cl}[1]{\mathcal{#1}}
\newcommand{\bb}[1]{\mathbb{#1}}

\begin{document}

\title[Spectral smoothings for compact quantum groups]{Convergence of Spectral Smoothings for Compact Quantum Group Actions}

\author[T. Camper]{Trevor Camper}
\address{Department of Mathematics\\
Dartmouth College\\ 209 Kemeny Hall\\ Hanover \\ NH 03755 \\ USA}
\email{trevor.camper@dartmouth.edu}

\author[D. Giannakis]{Dimitrios Giannakis}
\address{Department of Mathematics\\ Dartmouth College\\
337 Kemeny Hall\\ Hanover \\ NH 03755 \\ USA}
\email{dimitrios.giannakis@dartmouth.edu}

\author[G. Hoefer]{Gage Hoefer}
\address{Department of Mathematics \\ Dartmouth College \\ 314 Kemeny Hall \\ Hanover \\ NH 03755 \\ USA}\email{gage.hoefer@dartmouth.edu}

\date{\today}
\begin{abstract}
In this note, we study quantum Gromov-Hausdorff convergence of operator systems obtained using smoothings of ${\rm C}^{*}$-algebras, where our smoothings arise from the action of a compact quantum group on the ${\rm C}^{*}$-algebra. Specifically, we show how if $\mathbb{G}$ is a co-amenable compact quantum group acting on unital ${\rm C}^{*}$-algebra $\mathcal{A}$, and both are equipped with suitable quantum metrics, bounded approximate identities in $L^{1}(\mathbb{G})$ may be used for approximations of the original space $\mathcal{A}$ under the quantum Gromov-Hausdorff distance. We finish by presenting some applications to the convergence of spectral truncations of compact groups and connections to previous results. 
  
\end{abstract}
\maketitle



\section{Introduction}
Motivated by questions arising in high energy physics, M. A. Rieffel developed the notion of a ``quantum" compact metric space, and established the study of convergence for classes of these spaces with respect to a particular metric (see \cite{rieffel, rieffel_two}). This metric, defined as a noncommutative analogue of the Gromov-Hausdorff distance for compact metric spaces, is quite natural to consider in the study of noncommutative geometry and ${\rm C}^{*}$-algebras more broadly. Indeed: the original definition of a compact quantum metric space was inspired by the work of Connes' in \cite{Connes89}, in which he first considered spectral triples $(\mathcal{A}, H, D)$ where $\mathcal{A}$ is a $*$-algebra represented as bounded linear operators on a Hilbert space $H$, and $D$ is a (usually unbounded) self-adjoint operator on $H$ with compact resolvent such that the commutator $[D, a]$ is bounded for all $a \in \mathcal{A}$. Given a spectral triple, one can define a metric $\rho_{D}$ on the state space $S(\mathcal{A})$ via the formula
\begin{gather*}
    \rho_{D}(\phi, \psi) = \sup\{|\phi(a)-\psi(a)|: \; \|[D, a]\| \leq 1\}.
\end{gather*}
\noindent If $M$ is a compact Riemannian spin manifold, with $\mathcal{A} = C(M)$ and $D = D_{M}$ is the corresponding Dirac operator, then $\rho_{D}$ precisely recovers the standard Riemannian metric on $M$ (see \cite[Proposition 1]{Connes89}). Rieffel pointed out that--- using a similar approach--- if $(X, \rho)$ is a general compact metric space, by defining the Lipschitz seminorm $L_{\rho}$ on $\mathcal{A} = C(X)$ via 
\begin{gather*}
    L_{\rho}(f) = \sup\{|f(x)-f(y)|/\rho(x, y): \; x \neq y\},
\end{gather*}
\noindent the metric information of $(X, \rho)$ can be completely recovered via the induced metric
\begin{gather*}
    \rho_{L}(\phi, \psi) = \sup\{|\phi(f)-\psi(f)|: \; L_{\rho}(f) \leq 1\}
\end{gather*}
\noindent defined on $S(C(X))$, the state space of the commutative unital ${\rm C}^{*}$-algebra $C(X)$. Thus, a noncommutative analogue of a compact metric space should rely in some way on a metric $L$ defined over $\mathcal{A}$ (where $\mathcal{A}$ is a general unital ${\rm C}^{*}$-algebra) which plays the role of $L_{\rho}$ in the commutative case. When $\mathcal{A}$ is part of a spectral triple, $L$ is given by $L(a) = \|[D, a]\|$ as seen before. 

Under this framework, significant work has been done studying various compact quantum metric spaces and the properties of convergence for approximating spaces (see \cite{AgKaadKyed22, ConnesSuij21, dand_lizzi_mart, ges, kk, kk_two, Kerr03, Lat16, Lat22, Leimbach25, LeimSuij24, OzawaRieffel05, rieffel_two, rieffel_three, Suij21}, and \cite{Toyota23} for a non-exhaustive list). The approximating spaces in most of these settings are constructed as finite-rank approximations of the multiplier algebra of the underlying Hilbert space, where the projections arise from the eigenfunctions of an associated Dirac-type operator. These naturally correspond to a Toeplitz approximation of this multiplier algebra (see \cite{bs}). However, there are many classical ways to produce approximations using the spectral information of an operator. One of the most celebrated methods in spectral theory is to construct the so-called ``heat kernel" of the associated operator (see \cite{Evans22}). 

\indent In many circumstances, convolution against this heat kernel acts as an approximate identity in the small time limit (see for example, \cite{Evans22} in the case of the Laplacian) and thus smooth approximations of the multiplier algebra can be obtained in this manner. In the case of more general ${\rm C}^{*}$-algebras, an appropriate form of convolution must be introduced. In 1984, Werner introduced the \textit{quantum harmonic analysis (QHA)} which provides a notion of convolution between functions and bounded operators when translation is available on the underlying Hilbert space (see \cite{Werner84}). QHA has seen a recent resurgence in the literature; see, for instance, \cite{FulscheGalke25} and the references therein for recent papers on its development, specifically in the case when the Hilbert space is $L^{2}(G, \mu_{G})$ for a locally compact, abelian group $G$. 

\indent Motivated by convolution against the heat kernel and QHA more generally, we construct families of operator systems which act as approximations for a unital ${\rm C}^{*}$-algebra, dependent on the (ergodic) action of a co-amenable compact quantum group on it. More specifically, using the analogue of function-operator convolution in the setting of compact quantum groups and their actions on ${\rm C}^{*}$-algebras, we construct operator systems converging to a target ${\rm C}^{*}$-algebra $\mathcal{A}$ using convolution against bounded approximate identities in $L^{1}(\mathbb{G})$, where $\mathbb{G}$ is the co-amenable compact quantum group with $\mathbb{G} \curvearrowright \mathcal{A}$. This provides an alternative construction for approximating spaces similar to (but distinct from) recent work in the literature (see, for instance \cite{ges, Leimbach25, Li09, rieffel_three, Toyota23} for a few examples). We provide an extensive discussion on how our results generalize the framework put forth in \cite{Werner84} for classical compact groups $G$ acting on unital ${\rm C}^{*}$-algebras. Additionally, we show how in the classical setting the convergence can be obtained through approximate identities which are also finite-rank integral operators on $L^{2}(G, \mu_{G})$, which recovers the classical Toeplitz approximations of the multiplier algebra. In particular, our results generalize the results recently obtained in \cite{ges}, which proved quantum Gromov-Hausdorff convergence of the state spaces of $C(G)$.

The article is organized as follows. Section \ref{sec:preliminaries} is preliminary, and meant to fix terminology and summarize results covering operator systems and compact quantum metric spaces, the quantum Gromov-Hausdorff distance, compact quantum groups, and their actions. Section \ref{sec:approximations_from_cqg_actions} includes our main results, discussing convergence of smoothings of ${\rm C}^{*}$-algebras obtained through ergodic actions of co-amenable compact quantum groups. Section \ref{sec:examples} provides a few natural examples which arise, and touches on connections with other work.


\section{Preliminaries}\label{sec:preliminaries}

\subsection{Operator systems and compact quantum metric spaces} Let $H, K, L$ denote Hilbert spaces; all inner products are assumed to be linear in the second variable. For $n \in \mathbb{N}$ we write $M_{n}$ to denote the $n\times n$ matrices over $\mathbb{C}$. We write $\mathcal{B}(H, K)$ for the space of all bounded linear operators from $H$ into $K$ and set $\mathcal{B}(H) := \mathcal{B}(H, H)$ for the ${\rm C}^{*}$-algebra of all bounded linear operators on $H$, and $I$ for the identity operator on $H$, when the context is clear. If $(X, \|\cdot\|)$ is a normed space, for $\delta > 0$ we let
\begin{gather*}
    \mathcal{B}_{\delta}^{X}(0) := \{x \in X: \; \|x\| \leq \delta\}
\end{gather*}
\noindent denote the closed ball of radius $\delta$ in $X$. If $(\Omega, \Sigma, \mu)$ is a $\sigma$-finite measure space, we write $1_{\Omega}: \Omega\rightarrow \mathbb{R}$ for the constant function $1_{\Omega}(x) = 1$ when $x \in \Omega$. We denote the \textit{flip map} by
\begin{gather*}
    \Sigma: H\otimes K \rightarrow K\otimes H, \;\;\;\; \xi\otimes \eta \mapsto \eta\otimes \xi.
\end{gather*}
\noindent Generally, if $x \in \mathcal{B}(H, K)$ we let ${\rm Ad}_{x}: \mathcal{B}(H)\rightarrow \mathcal{B}(K)$ denote the mapping $y \mapsto xyx^{*}$. We also will use the \textit{leg numbering notation}: if $H, K, L$  are as before and $x \in \mathcal{B}(H\otimes K)$, then
\begin{gather*}
    x_{12} := x\otimes I, \;\;\;\; x_{23} := I\otimes x, \;\;\;\; x_{13} := ({\rm id}\otimes {\rm Ad}\Sigma)x_{12} = ({\rm Ad}\Sigma \otimes {\rm id})x_{23},
\end{gather*}
\noindent where we assume $I$ (and ${\rm id}$) are acting without ambiguity on their respective spaces.

An \textit{operator space} $\mathcal{X}$ is a norm-closed subspace of $\mathcal{B}(H, K)$. We write $M_{n}(\cl{X})$ to denote the space of $n\times n$ matrices whose entries are in $\cl{X}$; note that this can canonically be given an operator space structure, where the norm is inherited from $\cl{X}$. An \textit{operator system} is an operator space $\mathcal{X} \subseteq \mathcal{B}(H)$ with $I_{H} \in \mathcal{X}$ and $\mathcal{X}^{*} = \mathcal{X}$. For a space $\mathcal{X}$ of operators inside $\mathcal{B}(H)$, we will write $\mathcal{X}'$ for their commutant, i.e.,
\begin{gather*}
    \mathcal{X}' := \{a \in \mathcal{B}(H): \; ab = ba \; {\rm for \; all \;} b \in \mathcal{X}\}.
\end{gather*}

Let $\mathcal{X}, \mathcal{Y}$ be operator spaces. For a linear map $\Phi: \cl{X}\rightarrow \cl{Y}$ and $n \in \mathbb{N}$, the $n^{\rm th}$ \textit{amplification} $\Phi^{(n)}: M_{n}(\cl{X})\rightarrow M_{n}(\cl{Y})$ is defined by
\begin{gather*}
    \Phi^{(n)}((x_{ij})_{i, j}) := (\Phi(x_{ij}))_{i, j}.
\end{gather*}
\noindent We say that $\Phi$ is \textit{completely bounded} if
\begin{gather*}
    \|\Phi\|_{\rm cb} := \sup\limits_{n \in \mathbb{N}}\|\Phi^{(n)}\| < \infty.
\end{gather*}
\noindent It is called \textit{completely contractive} if $\|\Phi\|_{\rm cb} \leq 1$, and \textit{completely isometric} if $\Phi^{(n)}$ is an isometry for every $n \in \mathbb{N}$. 

Now, let $\mathcal{X}, \mathcal{Y}$ be operator systems. We say that $\Phi: \mathcal{X}\rightarrow \mathcal{Y}$ is \textit{unital} if $\Phi(1_{\cl{X}}) = 1_{\cl{Y}}$, and \textit{positive} if $\Phi(\mathcal{X}_{+}) \subseteq \mathcal{Y}_{+}$. The map $\Phi$ is called \textit{completely positive} if $\Phi^{(n)}$ is positive for each $n \in \mathbb{N}$. We say that $\Phi$ is a \textit{complete order isomorphism} if it is a completely positive map that is also a bijection and its inverse is completely positive. We write CP (respectively, UCP) for completely positive (respectively, unital completely positive) maps $\Phi: \cl{X}\rightarrow \cl{Y}$. A \textit{state} on $\mathcal{X}$ is a positive linear functional $\phi: \cl{X}\rightarrow \mathbb{C}$; it is known that every state on $\mathcal{X}$ is automatically completely positive, with norm $1$ (see \cite[Proposition 2.11, Chapter 3]{paulsen}). We let $S(\mathcal{X})$ denote the set of all states on $\mathcal{X}$; by the previous remarks and the Banach-Alaoglu Theorem, the state space $S(\mathcal{X})$ becomes a compact Hausdorff space when endowed with the ${\rm weak}^{*}$-topology. Similarly, for $n \in \mathbb{N}$ we let $S_{n}(\cl{X})$ denote the set of UCP maps $\Phi: \cl{X}\rightarrow M_{n}$.

We write $\mathcal{X}_{h}$ for the set of all self-adjoint elements in $\mathcal{X}$; that is, $x \in \mathcal{X}_{h}$ if $x^{*} = x$. Additionally, we write $\mathcal{X}_{+}$ for the positive elements in $\mathcal{X}$ (where positivity is inherited from the ambient ${\rm C}^{*}$-algebra). Note that every $x \in \mathcal{X}_{h}$ may be written as a difference of elements in $\mathcal{X}_{+}$, using the decomposition
\begin{gather*}
    x = \frac{1}{2}(\|x\|1_{\cl{X}}+x)-\frac{1}{2}(\|x\|1_{\cl{X}}-x),
\end{gather*}
\noindent which can be used to show that any positive map $\Phi: \cl{X}\rightarrow \cl{Y}$ automatically satisfies $\Phi(x^{*}) = \Phi(x)^{*}$. For $x \in \cl{X}$, we let
\begin{gather*}
    {\rm Re}(x) := \frac{x+x^{*}}{2}, \;\;\;\; {\rm Im}(x) := \frac{x-x^{*}}{2i}, \;\;\;\; {\rm Re}(x), \; {\rm Im}(x) \in \cl{X}_{h}
\end{gather*}
\noindent with $x = {\rm Re}(x)+i{\rm Im}(x)$.

In this work, we will take the operator system approach to compact quantum metric spaces; we follow the exposition in \cite{kk}, and refer the interested reader to \cite{ConnesSuij21, Leimbach25, Li03, Li06} for similar approaches. A seminorm $L: \cl{X}\rightarrow [0, \infty]$ is called a \textit{Lip-norm} if
\begin{itemize}
    \item[(i)] $L$ is \textit{densely-defined}, in that ${\rm dom}(L) := \{x \in \cl{X}: \; L(x) < \infty\}$ is norm-dense in $\cl{X}$;
    \item[(ii)] $\ker(L) = \mathbb{C}1_{\cl{X}}$;
    \item[(iii)] $L(x^{*}) = L(x)$ for all $x \in \cl{X}$;
    \item[(iv)] the topology on $S(\cl{X})$ induced from the metric
    \begin{gather*}
        d^{L}(\mu, \nu) := \sup\{|\mu(x)-\nu(x)|: \; L(x) \leq 1\}
    \end{gather*}
    \noindent is the ${\rm weak}^{*}$-topology. 
\end{itemize}
\noindent A \textit{compact quantum metric space} is a pair $(\mathcal{X}, L)$ consisting of an operator system $\mathcal{X} \subseteq \mathcal{B}(H)$ and a Lip-norm $L$ on $\mathcal{X}$. If $\cl{X}$ is an operator system endowed with a seminorm $L: \cl{X}\rightarrow [0, \infty]$ satisfying (i)-(iii) above, we let
\begin{gather*}
    d^{L, n}(\phi, \psi) := \sup\limits_{x \in \cl{X}/\mathbb{C}1_{\cl{X}}}\frac{\|\phi(x)-\psi(x)\|}{L(x)}
\end{gather*}
\noindent denote the induced distance on $S_{n}(\cl{X})$. 
\begin{remark}
\rm While the original definitions \cite[Definition 2.1.1, 2.1.2]{kk} allow for greater flexibility in what is considered a Lip-norm by allowing $\mathbb{C}1_{\cl{X}} \subseteq \ker(L)$ alone, the remarks surrounding the aforementioned definitions show that if $\ker(L) = \mathbb{C}1_{\cl{X}}$ precisely then $d^{L}$ on $S(\cl{X})$ as given above is a well-defined, finite metric. We may without loss of generality work with the definition of a Lip-norm given above. 
\end{remark}

\begin{remark}\label{rem:aou_vs_op_system}
\rm Technically, a compact quantum metric space was defined only for real partially ordered vector spaces, as first introduced in \cite{rieffel_two}. As is well-known in the literature, we may relate the two in a natural way. If we consider a norm-dense real subspace $V \subseteq \mathcal{X}_{h}$ (which is itself a real subspace) with $1_{\cl{X}} \in V$, a seminorm $L^{0}: V\rightarrow [0, \infty)$ such that $L^{0}(1_{\cl{X}}) = 0$ is called an \textit{order-unit Lip-norm} if $d^{L^{0}}$ metrizes the ${\rm weak}^{*}$-topology on $S(\cl{X})$, with $(V, L^{0})$ an \textit{order-unit compact quantum metric space}. This can be lifted to a Lip-norm $L_{\rm os}^{0}$ on $\cl{X}$ by defining
\begin{gather*}
    L_{\rm os}^{0}(x) := \begin{cases}
        \sup\limits_{\theta \in [0, 2\pi]}L^{0}(\cos(\theta){\rm Re}(x)+\sin(\theta){\rm Im}(x), \;\;\;\; {\rm Re}(x), {\rm Im}(x) \in {\rm dom}(L^{0}),
        \\ \infty, \;\;\;\;\;\;\;\;\;\;\;\;\;\;\;\;\;\;\;\;\;\;\;\;\;\;\;\;\;\;\;\;\;\;\;\;\;\;\;\;\;\;\;\;\;\;\;\;\;\;\;\;\;\;\;\;\; {\rm otherwise}.
    \end{cases}
\end{gather*}
\noindent Conversely, if $L: \cl{X}\rightarrow [0, \infty]$ is a Lip-norm on operator system $\cl{X}$, we can define an order-unit Lip-norm $L_{h}: {\rm dom}(L_{h})\rightarrow [0, \infty)$ by restricting $L$ to ${\rm dom}(L_{h}) := {\rm dom}(L) \cap L_{h}$. We note that $(L_{\rm os}^{0})_{h} = L^{0}$, making the two approaches compatible. That these yield appropriate compact quantum metric spaces is the statement of \cite[Proposition 2.1.8]{kk}. 
\end{remark}

For a compact quantum metric space $(\mathcal{X}, L)$, let 
\begin{gather*}
    B_{1}^{L}(0) := \{x \in \mathcal{X}: \; L(x) \leq 1\},
    \\ \mathcal{L}_{1}^{L} := \{x \in \mathcal{X}: \; x \in B_{1}^{L}(0), \; \|x\| \leq1\},
\end{gather*}
\noindent where $\|x\|$ in the latter definition denotes the norm as an element $x \in \mathcal{X}$. As $\cl{X}$ is a normed space, let $q: \cl{X}\rightarrow \cl{X}/\mathbb{C}$ denote the canonical quotient map to the normed quotient space $\cl{X}/\mathbb{C}$. 
For a seminorm $L$ on an operator system $\cl{X}$, an equivalent way to check that $L$ is a Lip-norm is by verifying (i)-(iii) above, in addition to the following:
\begin{itemize}
    \item[(iv')] the set $q(B_{1}^{L}(0))$, as defined above, is totally bounded in $\cl{X}/\mathbb{C}$ with respect to the quotient norm $\|\cdot\|_{\cl{X}/\mathbb{C}}$ on $\cl{X}/\mathbb{C}$.
\end{itemize}
\noindent This follows from \cite[Theorem 1.8]{rieffel} (and is in fact and if and only if statement); we rely on this characterization extensively. 

Given compact quantum metric spaces $(\cl{X}, L)$ and $(\hat{\cl{X}}, \hat{L})$, a seminorm $M: \cl{X}\oplus \hat{\cl{X}}\rightarrow [0, \infty]$ is said to be \textit{admissible} if $M$ is a Lip-norm on $\cl{X}\oplus \hat{\cl{X}}$, ${\rm dom}(M) = {\rm dom}(L)\oplus {\rm dom}(\hat{L})$ and the quotient seminorms induced by $M_{h}$ via the coordinate projections
\begin{gather*}
    {\rm dom}(M)_{h}\rightarrow {\rm dom}(L)_{h}, \;\;\;\; {\rm dom}(M)_{h}\rightarrow {\rm dom}(\hat{L})_{h},
\end{gather*}
\noindent agree with $L_{h}$ and $\hat{L}_{h}$, respectively. We let $\cl{M}(L, \hat{L})$ denote the collection of all admissible Lip-norms corresponding to $(\cl{X}, L)$ and $(\hat{\cl{X}}, \hat{L})$. When $M$ is an admissible Lip-norm, it follows that coordinate projections $\cl{X}\oplus \hat{\cl{X}}\rightarrow \cl{X}$ and $\cl{X}\oplus \hat{\cl{X}}\rightarrow \hat{\cl{X}}$ induce isometric inclusions $S(\cl{X}), S(\hat{\cl{X}}) \hookrightarrow S(\cl{X}\oplus \hat{\cl{X}})$, and allows us to measure the Hausdorff distance between $S(\cl{X})$ and $S(\hat{\cl{X}})$ with respect to the induced metric $d^{M}$ in $S(\cl{X}\oplus \hat{\cl{X}})$. Then the \textit{quantum Gromov-Hausdorff distance} between $(\cl{X}, L)$ and $(\hat{\cl{X}}, \hat{L})$ is given by
\begin{gather*}
    {\rm d}_{q}((\cl{X}, L), (\hat{\cl{X}}, \hat{L})) := \inf\{{\rm d}^{M}_{H}(S(\cl{X}), S(\hat{\cl{X}})): \; M \in \cl{M}(L, \hat{L})\},
\end{gather*}
\noindent where ${\rm d}_{H}^{M}$ denotes the Hausdorff distance between state spaces $S(\cl{X}), S(\hat{\cl{X}})$ inside $S(\cl{X}\oplus \hat{\cl{X}})$ with respect to metric $d^{M}$. 

Similarly, following \cite{Kerr03} (see also \cite{KerrLi09}, \cite{Leimbach25}, \cite{Li06}) for $n \in \mathbb{N}$ we set
\begin{gather*}
    {\rm d}_{n}^{\rm s}((\cl{X}, L), (\hat{\cl{X}}, \hat{L})) := \inf {\rm d}_{H}^{d^{L, n}}(S_{n}(\cl{X}), S_{n}(\hat{\cl{X}}))
\end{gather*}
\noindent where the infimum is taken over all admissible Lip-norms $L$ on $\cl{X}\oplus \hat{\cl{X}}$. We may then define the \textit{complete Gromov-Hausdorff distance} between $(\cl{X}, L)$ and $(\hat{\cl{X}}, \hat{L})$ as
\begin{gather*}
    {\rm d}^{\rm s}((\cl{X}, L), (\hat{\cl{X}}, \hat{L})) := \inf\sup\limits_{n \in \mathbb{N}}{\rm d}_{H}^{d^{L, n}}(S_{n}(\cl{X}), S_{n}(\hat{\cl{X}})),
\end{gather*}
\noindent where the infimum is taken over all admissible Lip-norms $L$ on $\cl{X}\oplus \hat{\cl{X}}$.

\begin{remark}\label{r_qgh_dominates_gh}
\rm Let $(\cl{X}, L)$ and $(\hat{\cl{X}}, \hat{L})$ be compact quantum metric spaces. Then 
\begin{gather*}
    {\rm d}_{\rm GH}((S(\cl{X}), d^{L}), (S(\hat{\cl{X}}), d^{\hat{L}})) \leq {\rm d}_{\rm q}((\cl{X}, L), (\hat{\cl{X}}, \hat{L})).
\end{gather*}
\noindent This can be seen directly from the discussion surrounding (2.1) of \cite{kk_two}. Additionally, as discussed surrounding \cite[Definition 3.2]{Kerr03} and noting that ${\rm d}_{\rm q}((\cl{X}, L), (\hat{\cl{X}}, \hat{L})) = {\rm d}_{1}^{\rm s}((\cl{X}, L), (\hat{\cl{X}}, \hat{L}))$ we have that
\begin{gather*}
    {\rm d}_{\rm q}((\cl{X}, L), (\hat{\cl{X}}, \hat{L}))  \leq {\rm d}_{m}^{\rm s}((\cl{X}, L), (\hat{\cl{X}}, \hat{L})) \leq {\rm d}_{n}^{\rm s}((\cl{X}, L), (\hat{\cl{X}}, \hat{L})) \leq {\rm d}^{s}((\cl{X}, L), (\hat{\cl{X}}, \hat{L})),
\end{gather*}
\noindent for $m, n \in \mathbb{N}$ with $m \leq n$. 
\end{remark}

For compact quantum metric spaces $(\cl{X}, L)$ and $(\hat{\cl{X}}, \hat{L})$, a \textit{morphism} between them is a UCP map $\Phi: \cl{X}\rightarrow \hat{\cl{X}}$ for which there exists a constant $C \geq 0$ such that $\hat{L}(\Phi(x)) \leq CL(x)$ for all $x \in \cl{X}$. We call the morphism \textit{$C^{1}$-contractive} if $\hat{L}(\Phi(x)) \leq L(x)$ for all $x \in \cl{X}$.

Following the example of \cite{LeimSuij24, Suij21}, we use the following terminology for approximating maps between compact quantum metric spaces. 

\begin{definition}
Let $\{(\cl{X}_{\Lambda}, L_{\Lambda})\}_{\Lambda \in I}$ be a net of compact quantum metric spaces, and let $(\cl{X}, L)$ be another compact quantum metric space. An {\rm approximate order isomorphism} for these spaces are pairs of linear maps $\Phi_{\Lambda}: \cl{X}\rightarrow \cl{X}_{\Lambda}$ and $\Psi_{\Lambda}: \cl{X}_{\Lambda}\rightarrow \cl{X}$ for any $\Lambda \in I$ such that 
\begin{itemize}
    \item[(i)] the maps $\Phi_{\Lambda}, \Psi_{\Lambda}$ are compact quantum metric space morphisms;
    \item[(ii)] there exist nets $(\gamma_{\Lambda})_{\Lambda \in I}$ and $(\gamma_{\Lambda}')_{\Lambda \in I}$ with $\gamma_{\Lambda}, \gamma_{\Lambda}' > 0$ converging to zero such that
    \begin{gather*}
       \|\Psi_{\Lambda}\circ \Phi_{\Lambda}(x)-x\| \leq \gamma_{\Lambda}\cdot L(x),
       \\ \|\Phi_{\Lambda}\circ \Psi_{\Lambda}(y)-y\| \leq \gamma_{\Lambda}'\cdot L_{\Lambda}(y),
    \end{gather*}
    \noindent for all $x \in \cl{X}, y \in \cl{X}_{\Lambda}$. 
\end{itemize}
\end{definition}
\noindent An approximate order isomorphism $((\Phi_{\Lambda}, \Psi_{\Lambda}))_{\Lambda \in I}$ will be called a \textit{$C^{1}$-approximate order isomorphism} if maps $\Phi_{\Lambda}, \Psi_{\Lambda}$ are $C^{1}$-contractive for all $\Lambda \in I$. The results in \cite[Proposition 5.19]{Leimbach25} show that $C^{1}$-approximate order isomorphisms are the right morphisms to use when approximating a compact quantum metric space in the complete Gromov-Hausdorff distance. We record this fact below, along with a consequence.
\begin{proposition}\label{prop_qgh_distance_appoi}
Assume $(\cl{X}_{\Lambda}, L_{\Lambda})$ and $(\cl{X}, L)$ are compact quantum metric spaces. If $((\Phi_{\Lambda}, \Psi_{\Lambda}))_{\Lambda \in I}$ is a $C^{1}$-approximate order isomorphism for $(\cl{X}_{\Lambda}, L_{\Lambda})$ and $(\cl{X}, L)$, then $(\cl{X}_{\Lambda}, L_{\Lambda})$ converges to $(\cl{X}, L)$ in both the complete and the quantum Gromov-Hausdorff distance.  
\end{proposition}
\begin{proof}
As $\Phi_{\Lambda}, \Psi_{\Lambda}$ are $C^{1}$-contractive, we have that $C_{\Phi} = C_{\Psi} = 1$ in the proof of \cite[Proposition 5.19]{Leimbach25}. Thus, in the same proof we have that ${\rm d}^{\rm s}((\cl{X}, L), (\cl{X}_{\Lambda}, L_{\Lambda})) \leq \max\{\gamma_{\Lambda}, \gamma_{\Lambda}'\}$, with both going to zero along the net. By Remark \ref{r_qgh_dominates_gh}, this shows convergence in ${\rm d}_{\rm q}$ as well. 
\end{proof}

\begin{corollary}\label{cor_qgh_implies_gh}
Let $(\cl{X}_{\Lambda}, L_{\Lambda})$ and $(\cl{X}, \Lambda)$ be compact quantum metric spaces. If $((\Phi_{\Lambda}, \Psi_{\Lambda}))_{\Lambda \in I}$ is a $C^{1}$-approximate order isomorphism for $(\cl{X}_{\Lambda}, L_{\Lambda})$ and $(\cl{X}, L)$, then $(S(\cl{X}_{\Lambda}), d^{L_{\Lambda}})$ converges to $(S(\cl{X}), d^{L})$ in the Gromov-Hausdorff distance.
\end{corollary}
\begin{proof}
This follows directly from Remark \ref{r_qgh_dominates_gh}.
\end{proof}

\subsection{Hopf algebras, compact quantum groups, and actions}\label{ss:hopf_alg_cqg}
In what follows, $\odot$ denotes the algebraic tensor product of vector spaces (often with additional algebraic structure), $\otimes$ denotes the minimal tensor product of ${\rm C}^{*}$-algebras, and $\overline{\otimes}$ denotes the von Neumann algebraic tensor product. In the setting of operator systems, $\otimes_{\rm min}$ denotes the minimal operator system tensor product, which is obtained by taking the closure of $\cl{X}\odot \cl{Y}$ inside $\mathcal{B}(H\otimes K)$ for concrete representations $\cl{X}\subseteq \cl{B}(H), \cl{Y}\subseteq \cl{B}(K)$. We refer the reader to \cite{deCommer17, Timmermann08} for a more detailed exposition.
\smallskip

A  \emph{Hopf $*$-algebra} is a quadruple $(\cl A,\Delta,\varepsilon,S)$ consisting of a unital $*$-algebra $\cl A$, a unital $*$-homomorphism $\Delta:\cl A\to \cl A\otimes \cl A$  such that $(\mathrm{id}\otimes \Delta)\Delta=(\Delta\otimes \mathrm{id})\Delta$ and  morphisms  $\varepsilon:\cl A\to\bb C$ (the \emph{counit}) and $S:\cl A\to \cl A$ (the \emph{antipode}) such that, for all $a\in \cl A$,
\begin{itemize}
  \item[(i)] $(\varepsilon\otimes \mathrm{id})\Delta(a)=(\mathrm{id}\otimes \varepsilon)\Delta(a)=a$;
  \item[(ii)] $m(S\otimes \mathrm{id})\Delta(a)=m(\mathrm{id}\otimes S)\Delta(a)=\varepsilon(a)\,1$,
\end{itemize}
where $m:\cl A\otimes \cl A\to \cl A$ is the multiplication map. If $\cl{A}, \cl{B}$ are unital $*$-algebras (not necessarily Hopf $*$-algebras), and $\phi: \cl{A}\rightarrow \mathbb{C}, \psi: \cl{B}\rightarrow \mathbb{C}$ are linear maps, we can define \textit{slice maps}
\begin{gather*}
    \phi \odot {\rm id}_{\cl{B}}: \cl{A}\odot \cl{B}\rightarrow \cl{B}, \;\;\;\; a\odot b \mapsto \phi(a)b,
    \\ {\rm id}_{\cl{A}}\odot \psi: \cl{A}\odot \cl{B}\rightarrow \cl{A}, \;\;\;\; a\odot b \mapsto \psi(b)a,
\end{gather*}
\noindent for which 
\begin{gather*}
    (\psi\circ (\phi\odot {\rm id}_{\cl{B}}))(x) = (\phi\circ ({\rm id}_{\cl{A}}\odot \psi))(x) = (\phi\otimes \psi)(x),\\
    b((\phi\odot {\rm id}_{\cl{B}})(x))b' = (\phi\odot {\rm id}_{\cl{B}})((1_{\cl{A}}\odot b)x(1_{\cl{A}}\odot b')),\\
    a(({\rm id}_{\cl{A}}\odot \psi)(x))a' = ({\rm id}_{\cl{A}}\odot \psi)((a\odot 1_{\cl{B}})x(a'\odot 1_{\cl{B}})),
\end{gather*}
\noindent holds for all $a, a' \in \cl{A}, b, b' \in \cl{B}$ and $x \in \cl{A}\odot \cl{B}$. Additionally, if $\phi, \psi$ are $*$-linear, so is $\phi\odot \psi$. If we also assume $\mathcal{A}, \mathcal{B}$ are ${\rm C}^{*}$-algebras (and not only unital $*$-algebras), when $\phi, \psi$ are continuous the maps $\phi\odot {\rm id}_{\cl{B}}$ and ${\rm id}_{\cl{A}}\odot \psi$ extend uniquely to norm-continuous maps $\phi\otimes {\rm id}_{\cl{B}}$ and ${\rm id}_{\cl{A}}\otimes \psi$ on the minimal tensor product $\cl{A}\otimes \cl{B}$ (see \cite[Chapter 12]{Timmermann08}). An analogous statement holds for operator systems: if $\cl{X}, \cl{Y}$ are operator systems with $\phi: \cl{X}\rightarrow \mathbb{C}, \psi: \cl{Y}\rightarrow \mathbb{C}$ completely positive linear functionals, the map $\phi\odot \psi: \cl{X}\odot \cl{Y}\rightarrow \mathbb{C}$ extends uniquely to a completely positive map $\phi\otimes \psi: \cl{X}\otimes_{\rm min} \cl{Y}\rightarrow \mathbb{C}$ for which
\begin{gather*}
    \phi\circ({\rm id}_{\cl{X}}\otimes \psi) = \psi\circ(\phi\otimes {\rm id}_{\cl{Y}}) = \phi\otimes \psi.
\end{gather*}
\noindent For a proof, see \cite[Lemma 2.3]{Leimbach25} (see also \cite[Theorem 12.3]{paulsen}). 
\smallskip

A \textit{compact quantum group (CQG)} $\bb{G}$ is a pair $(C(\bb{G}), \Delta)$, where $C(\bb{G})$ is a unital ${\rm C}^{*}$-algebra and $\Delta: C(\bb{G})\rightarrow C(\bb{G})\otimes C(\bb{G})$ is a unital $*$-homomorphism such that:
\begin{itemize}
  \item[(i)] $(\mathrm{id}\otimes \Delta)\Delta = (\Delta\otimes \mathrm{id})\Delta$;
  \item[(ii)] we have the following density conditions:
  \[
    \overline{\mathrm{span}}\bigl\{ \Delta(C(\bb{G}))(1\otimes C(\bb{G})) \bigr\}
    \;=\; C(\bb{G})\otimes C(\bb{G})
    \;=\;
    \overline{\mathrm{span}}\bigl\{ (C(\bb{G})\otimes 1)\Delta(C(\bb{G})) \bigr\}.
  \]
\end{itemize}

Every compact quantum group admits a unique state $\varphi_{\bb{G}}: C(\bb{G})\rightarrow \bb{C}$, called the \textit{Haar state}, such that
\[
(\mathrm{id}\otimes \varphi_{\bb{G}})\Delta(a) = \varphi_{\bb{G}}(a)\,1
\;=\;
(\varphi_{\bb{G}}\otimes \mathrm{id})\Delta(a)
\qquad (a \in C(\bb{G})).
\]
If the Haar state is tracial, then $\bb{G}$ is said to be of \textit{Kac type}. If the Haar state $\varphi_{\bb{G}}$ is faithful, then $\mathbb{G}$ is called \textit{reduced}. 

Given two compact quantum groups $\bb{G} = (C(\bb{G}), \Delta_{\bb{G}})$ and $\bb{H} = (C(\bb{H}), \Delta_{\bb{H}})$, a unital $*$-homomorphism $\varphi: C(\bb{G})\rightarrow C(\bb{H})$ is a \emph{morphism of compact quantum groups} if
\[
(\varphi\otimes \varphi)\circ \Delta_{\bb{G}} \;=\; \Delta_{\bb{H}}\circ \varphi.
\]
If $\varphi$ is surjective, we write $\bb H \subseteq \bb G$ and call $\bb H$ a subgroup of $\bb G$; if $\varphi$ is injective, we call $\bb H$ a quotient of $\bb G$; and if $\varphi$ is bijective, then $\bb{G}$ and $\bb{H}$ are \emph{isomorphic as compact quantum groups}.

A \textit{unitary finite-dimensional representation} $\pi$ of $\mathbb{G}$ consists of a pair $(H_{\pi}, U_{\pi})$ where $H_{\pi}$ is a finite-dimensional Hilbert space, and $U_{\pi} \in \mathcal{B}(H_{\pi})\otimes C(\bb{G})$ is a unitary operator such that $({\rm id}_{\pi}\otimes \Delta)(U_{\pi}) = U_{\pi, 12}U_{\pi, 13}$ in $\mathcal{B}(H_{\pi})\otimes C(\bb{G})\otimes C(\bb{G})$. Given $\xi, \eta \in H_{\pi}$, we define the matrix coefficient $U_{\pi}(\xi, \eta) := (\omega_{\xi, \eta}\otimes {\rm id}_{C(\bb{G})})(U_{\pi})$, where $\omega_{\xi, \eta}: \mathcal{B}(H_{\pi})\rightarrow \mathbb{C}$ is the vector functional given by $\omega_{\xi, \eta}(x) := \langle \xi, x\eta\rangle$ for all $x \in \mathcal{B}(H_{\pi})$. If $n_{\pi}$ denotes the dimension of $H_{\pi}$ for a fixed unitary finite-dimensional representation $\pi$, and if $\{e_{i}^{\pi}\}_{i=1}^{n_{\pi}}$ is an orthonormal basis for $H_{\pi}$, then
\begin{gather}\label{eqn:comult_on_matrix_coeff}
    \Delta(U_{\pi}(\xi, \eta)) = \sum\limits_{i=1}^{n_{\pi}}U_{\pi}(\xi, e_{i}^{\pi})\otimes U_{\pi}(e_{i}^{\pi}, \eta).
\end{gather}
\noindent For two unitary finite-dimensional representations $\pi, \rho$, a linear map $T: H_{\pi}\rightarrow H_{\rho}$ is called an \textit{intertwiner} if $(T\otimes 1)U_{\pi} = U_{\rho}(T\otimes 1)$; we write ${\rm Mor}(\pi, \rho)$ for the set of all such intertwiners. A representation $\pi$ of $\bb{G}$ is said to be \textit{irreducible}, if ${\rm Mor}(\pi, \pi)$ is one-dimensional (i.e., if ${\rm Mor}(\pi, \pi) \cong \bb{C}I_{H_{\pi}}$). If $\pi, \rho$ are unitary finite-dimensional representations, we define $U_{\pi \boxtimes \rho} := U_{\pi, 13}U_{\rho, 23} \in \mathcal{B}(H_{\pi}\otimes H_{\rho})\otimes C(\bb{G})$ for the tensor product representation $\pi\boxtimes \rho$.

We fix, now and in the sequel, a maximal family ${\rm Irr}(\bb{G})$ of pairwise non-equivalent irreducible representations. We also set $\cl{O}(\bb{G})$ as the linear span of all matrix coefficients for $\bb{G}$. This is a norm-dense subalgebra of $C(\bb{G})$ on which $\varphi_{\bb{G}}$ is faithful, and $(\cl{O}(\bb{G}), \Delta, \varepsilon, S)$ forms a Hopf $*$-algebra with co-unit and antipode
\begin{gather*}
    \varepsilon(U_{\pi}(\xi, \eta)) = \langle \xi, \eta\rangle, \;\;\;\; S(U_{\pi}(\xi, \eta)) = U_{\pi}(\xi, \eta)^{*}.
\end{gather*}
\noindent One uses (\ref{eqn:comult_on_matrix_coeff}) to verify that $\Delta$ restricts to a well-defined co-multiplication on $\cl{O}(\bb{G})$.

We can use the $*$-algebra structure of $\cl{O}(\bb{G})$ to define a convolution product on $\cl{O}(\bb{G})^{*}$, the linear dual of $\cl{O}(\bb{G})$. Indeed: for $\omega, \rho \in \cl{O}(\bb{G})^{*}$ and any $x \in \cl{O}(\bb{G})$, define
\begin{equation}
    \begin{gathered}
        (\omega \ast \rho)(x) := (\omega \otimes \rho)\Delta(x), \;\;\;\; \omega^{*}(x) := \overline{\omega(S(x)^{*})}, 
    \\ \omega \ast x := ({\rm id}_{C(\mathbb{G})}\otimes \omega)\Delta(x), \;\;\;\; x \ast \omega := (\omega \otimes {\rm id}_{C(\mathbb{G})})\Delta(x).
    \end{gathered}\label{eqn:convolution_product_for_group}
\end{equation}
\noindent In this case, one may show (see \cite[Theorem 1.8]{deCommer17}) for each unitary finite-dimensional representation $\pi$ of $\mathbb{G}$ that there exists a unique invertible functional $f \in \cl{O}(\bb{G})^{*}$ such that:
\begin{itemize}
    \item[(i)] $Q_{\pi} = ({\rm id}_{H_{\pi}}\otimes f)U_{\pi}$ is positive;
    \item[(ii)] $f^{-1}\ast x \ast f = S^{2}(x)$ for all $x \in \cl{O}(\bb{G})$. 
\end{itemize}
\noindent The element $Q_{\pi}$ is an invertible element of $\mathcal{B}(H_{\pi})$, with 
\begin{gather*}
    U_{\pi}(\xi, \eta)^{*} = U_{\overline{\pi}}((Q_{\pi}^{-1}\xi)^{*}, \eta^{*}), \;\;\;\; S(U_{\pi}(\xi, \eta)) = U_{\overline{\pi}}((Q_{\pi}^{-1}\eta)^{*}, \xi^{*})
\end{gather*}
\noindent where $\overline{\pi}$ is the dual representation and $H_{\overline{\pi}} = H_{\pi}^{*}$ endowed with the inner product
\begin{gather*}
    \langle\langle \xi^{*}, \eta^{*}\rangle\rangle = \langle \eta, Q_{\pi}\xi\rangle. 
\end{gather*}

With respect to the Haar state $\varphi_{\bb{G}}$, we consider the GNS representation $(L^{2}(\bb{G}), \pi_{\bb{G}}, \xi_{\bb{G}})$ where $\pi_{\bb{G}}: C(\bb{G})\rightarrow \mathcal{B}(L^{2}(\bb{G}))$ is a $*$-representation and $\xi_{\bb{G}}\in L^{2}(\bb{G})$ is the cyclic vector such that $\varphi_{\bb{G}}(a) = \langle \xi_{\bb{G}}, \pi_{\bb{G}}(a)\xi_{\bb{G}}\rangle$ for all $a \in C(\bb{G})$. We also set $\Lambda: C(\bb{G})\rightarrow L^{2}(\bb{G})$ as the map sending $a \mapsto \pi_{\bb{G}}(a)\xi_{\bb{G}}$. Set $C_{r}(\mathbb{G}) := \pi_{\mathbb{G}}(C(\mathbb{G})) \subseteq \mathcal{B}(L^{2}(\mathbb{G}))$. Then $\Delta$ descends to a well-defined co-multiplication also denoted $\Delta$ on $C_{r}(\mathbb{G})$. The pair $(C_{r}(\bb{G}), \Delta)$ then defines what is called the \textit{reduced compact quantum group} $\bb{G}_{r}$. We have $\cl{O}(\bb{G}_{r}) = \pi_{\bb{G}}(\cl{O}(\bb{G}))$, and Haar state, co-unit and antipode satisfying 
\begin{gather*}
    \varphi_{\bb{G}_{r}}\circ \pi_{\bb{G}} = \varphi_{\bb{G}}, \;\;\;\; \varepsilon_{r}\circ \pi_{\bb{G}} = \varepsilon, \;\;\;\; S_{r}\circ \pi_{\bb{G}} = S.
\end{gather*}
\noindent One can show that for any $*$-representation of $\mathcal{O}(\mathbb{G})$ and any representation $\pi$ of $\mathbb{G}$, $\|\rho(U_{\pi}(\xi, \eta))\| \leq \|\xi\|\|\eta\|$; it follows that $\mathcal{O}(\mathbb{G})$ admits a universal ${\rm C}^{*}$-envelope, which we denote by $C_{u}(\mathbb{G})$. The ${\rm C}^{*}$-algebra $C_{u}(\mathbb{G})$ inherits the co-multiplication from $\mathcal{O}(\mathbb{G})$, making $\mathbb{G}_{u} := (C_{u}(\mathbb{G}), \Delta)$ a well-defined compact quantum group. Furthermore, by universal properties we have canonical surjective morphisms of compact quantum groups
\begin{gather*}
    C_{u}(\mathbb{G}) \twoheadrightarrow_{\pi_{u}} C(\mathbb{G})\twoheadrightarrow_{\pi_{\mathbb{G}}} C_{r}(\mathbb{G}).
\end{gather*}

\noindent There is a (canonical) associated von Neumann algebra $L^{\infty}(\bb{G}) := \pi_{\bb{G}}(\cl{O}(\bb{G}))'' \subseteq \cl{B}(L^{2}(\bb{G}))$ which carries the co-multiplication $\Delta: L^{\infty}(\mathbb{G})\rightarrow L^{\infty}(\mathbb{G})\overline{\otimes} L^{\infty}(\mathbb{G})$ extending the co-multiplication on $C_{r}(\bb{G})$. The vector state $\omega_{\xi_{\bb{G}}}: L^{\infty}(\bb{G})\rightarrow \bb{C}$ is faithful, and satisfies
\begin{gather*}
    (\omega_{\xi_{\bb{G}}}\otimes {\rm id}_{L^{\infty}(\bb{G})})\Delta = \omega_{\xi_{\bb{G}}}(\cdot)1 = ({\rm id}_{L^{\infty}(\bb{G})}\otimes \omega_{\xi_{\bb{G}}})\Delta.
\end{gather*}
\noindent Note that $C_{r}(\mathbb{G})$ is the norm-closure of $\pi_{\bb{G}}(\cl{O}(\bb{G}))$, while $L^{\infty}(\bb{G})$ is the weak-operator closure of $\pi_{\bb{G}}(\cl{O}(\bb{G}))$ inside $\cl{B}(L^{2}(\bb{G}))$. We set $L^{1}(\bb{G}) := L^{\infty}(\bb{G})_{\ast}$ as the pre-dual of the von Neumann algebra $L^{\infty}(\bb{G})$; this consists of all normal linear functionals $\omega: L^{\infty}(\bb{G})\rightarrow \bb{C}$.

A compact quantum group $\mathbb{G}$ is called \textit{co-amenable} if the co-unit $\varepsilon_{r}: \cl{O}(\bb{G}_{r})\rightarrow \bb{C}$ is norm-bounded, and thus extends to a $*$-homomorphism $\varepsilon_{r}: C_{r}(\bb{G})\rightarrow \bb{C}$. The results of \cite{BedosMurphTus01} show that $\bb{G}$ is co-amenable if and only if the Haar state $\varphi_{\bb{G}}$ is faithful and the co-unit $\varepsilon$ is norm-bounded; in this case, $\pi_{u}\circ \pi_{\mathbb{G}}$ is a $*$-isomorphism, and we have $C_{u}(\mathbb{G}) \cong C(\bb{G}) \cong C_{r}(\bb{G})$ $*$-isometrically as compact quantum groups, where the norm-boundedness of $\varepsilon$ is inherited from the norm-boundedness of $\varepsilon_{r}$ by assumption. Note that if $\bb{G}$ is assumed to be co-amenable, then using the $*$-isometric identification of $C(\bb{G})$ with $C_{r}(\bb{G})$, we may view any normal, continuous functional $\omega \in L^{1}(\bb{G}) = L^{\infty}(\bb{G})_{\ast}$ as a continuous functional $\hat{\omega}: C(\bb{G})\rightarrow \bb{C}$ via $\hat{\omega} := \omega\circ \pi_{\varphi_{\bb{G}}}$. Thus, we have an injective embedding $L^{1}(\bb{G}) \hookrightarrow C(\bb{G})^{*}$ for co-amenable compact quantum groups. With this embedding in mind, we say a \textit{bounded approximate identity of $L^{1}(\mathbb{G})$} is a net $(\omega_{\Lambda})_{\Lambda \in I} \subseteq L^{1}(\mathbb{G})$ such that:
\begin{itemize}
    \item[(i)] $\sup\limits_{\Lambda \in I}\|\omega_{\Lambda}\|_{L^{1}(\mathbb{G})} < \infty$;
    \item[(ii)] for every $\omega \in L^{1}(\mathbb{G})$, 
    \begin{gather*}
        \lim\limits_{\Lambda}\|\omega_{\Lambda}\ast \omega-\omega\|_{L^{1}(\mathbb{G})}\rightarrow 0 \;\;\;\; {\rm and} \;\;\;\; \lim\limits_{\Lambda}\|\omega\ast \omega_{\Lambda}-\omega\|_{L^{1}(\mathbb{G})}\rightarrow 0,
    \end{gather*}
    \noindent where $\ast$ is defined as in (\ref{eqn:convolution_product_for_group}). 
\end{itemize}
\begin{definition}\label{def:l_1_approx}
We call a bounded approximate identity $(\omega_{\Lambda})_{\Lambda \in I} \subseteq L^{1}(\mathbb{G})$ an {\rm $L^{1}(\mathbb{G})$-approximate identity} if, in addition:
\begin{itemize}
    \item[(i)] $\omega_{\Lambda} \geq 0$ for each $\Lambda \in I$;
    \item[(ii)] $\omega_{\Lambda}(1_{C(\mathbb{G})}) = 1$ for each $\Lambda \in I$;
    \item[(iii)] $\|\omega_{\Lambda}\|_{L^{1}(\mathbb{G})} = 1$ for each $\Lambda \in I$.
\end{itemize}
\end{definition}
\noindent Thus, an $L^{1}(\mathbb{G})$-approximate identity is a bounded approximate identity consisting of states. Co-amenability of $\mathbb{G}$ ensures the existence of at least one $L^{1}(\mathbb{G})$-approximate identity (see \cite[Theorem 2]{HuNeufangRuan10}). 

Let $\cl{A}$ be a unital ${\rm C}^{*}$-algebra, and let $\bb{G} = (C(\bb{G}), \Delta)$ be a compact quantum group. A \textit{(right) action} of $\bb{G}$ on $\cl{A}$ is a unital $*$-homomorphism $\alpha: \cl{A}\rightarrow \cl{A}\otimes C(\bb{G})$ such that:
\begin{itemize}
    \item[(i)] $(\alpha \otimes {\rm id}_{C(\bb{G})})\circ \alpha = ({\rm id}_{\cl{A}}\otimes \Delta)\circ \alpha$;
    \item[(ii)] the following density condition holds:
    \begin{gather*}
        [\alpha(\cl{A})(1_{\cl{A}}\otimes C(\bb{G}))] = \cl{A}\otimes C(\bb{G}).
    \end{gather*}
\end{itemize}
\noindent Such a $*$-homomorphism is called a \textit{co-action} of one unital $*$-algebra on another. If $\mathbb{G}$ acts on $\mathcal{A}$ via $\alpha$, we set 
\begin{gather*}
    \mathcal{A}^{\mathbb{G}} = \{a \in \cl{A}: \; \alpha(a) = a\otimes 1_{C(\mathbb{G})}\}
\end{gather*}
\noindent to be the fixed point subalgebra under the co-action. The action of $\mathbb{G}$ on $\mathcal{A}$ is called \textit{ergodic} if $\mathcal{A}^{\mathbb{G}} = \mathbb{C}1_{\cl{A}}$.  

\begin{remark}\label{rem:co-unit_identity}
\rm When $\mathbb{G}$ is co-amenable and acts on $\mathcal{A}$, the proof of \cite[Lemma 1.4.(a)]{ExelNg02} shows that $({\rm id}_{\mathcal{A}}\otimes \varepsilon)\alpha = {\rm id}_{\mathcal{A}}$, with $\alpha$ an injective unital $*$-homomorphism.
\end{remark}

Similar to the convolution product defined in (\ref{eqn:convolution_product_for_group}), for $\omega \in C(\mathbb{G})^{*}, \phi\in \mathcal{A}^{*}$ and $a \in \mathcal{A}$ define the elements
\begin{gather}\label{eqn:convolution_for_algebra}
    \omega \star a := ({\rm id}_{\mathcal{A}}\otimes \omega)\alpha(a), \;\;\;\; a\star \phi := (\phi\otimes {\rm id}_{C(\mathbb{G})})\alpha(a),
\end{gather}
\noindent which are in $\mathcal{A}$ and $C(\mathbb{G})$, respectively. 
\begin{remark}
\rm Note that if we view $\Delta: C(\mathbb{G})\rightarrow C(\mathbb{G})\otimes C(\mathbb{G})$ as a co-action of $\mathbb{G}$ on itself, then the product $\ast$ defined in (\ref{eqn:convolution_product_for_group}) coincide with product $\star$ defined in (\ref{eqn:convolution_for_algebra}).  
\end{remark}

\begin{remark}\label{rem:function_op_conv}
\rm To make the point explicit for the reader, suppose $G$ is a classical compact group with Haar measure $\mu_{G}$. Assume $G \curvearrowright \mathcal{A}$ is an action of $G$ on a unital ${\rm C}^{*}$-algebra $\mathcal{A} \subseteq \mathcal{B}(L^{2}(G, \mu_{G}))$ through $\sigma: G\rightarrow {\rm Aut}(\mathcal{A})$, where $\sigma$ denotes a homomorphism. We set $\mathbb{G} = (C(G), \Delta)$ where co-multiplication $\Delta: C(G)\rightarrow C(G\times G)$ is given by 
\begin{gather}\label{eqn:classical_co_multiplication}
    \Delta(f)(g, h) := f(gh), \;\;\;\; f \in C(G), \; g, h \in G.
\end{gather}
\noindent Note here we use the identification $C(G)\otimes C(G) \cong C(G\times G)$. If we consider $\mu_{G}: C(G)\rightarrow \bb{C}$ as a state on the unital, abelian ${\rm C}^{*}$-algebra of continuous functions over $G$, then $\mu_{G}$ is a faithful state corresponding precisely to the Haar state of $\mathbb{G}$, and the corresponding GNS representation yields the standard $L^{2}$-space $L^{2}(G, \mu_{G})$, with $\pi_{G}: C(G)\rightarrow \cl{B}(L^{2}(G, \mu_{G}))$ representing elements in $C(G)$ as multiplication operators. Additionally, $L^{\infty}(\bb{G}) = \pi_{G}(C(G))'' \cong L^{\infty}(G, \mu_{G})$, the space of all essentially bounded measurable functions on $G$, and $L^{1}(\mathbb{G}) = L^{\infty}(G, \mu_{G})_{\ast} = L^{1}(G, \mu_{G})$. On the dense Hopf $*$-algebra inside $C(G)$, $\varepsilon_{G}: \mathcal{O}(G)\rightarrow \mathbb{C}$ is given by $\varepsilon_{G}(f) := f(e)$, and $S_{G}: \mathcal{O}(G)\rightarrow \mathcal{O}(G)$ is given by $S_{G}(f)(g) := f(g^{-1})$ for $f \in \mathcal{O}(G)$.

The action of $G\curvearrowright \mathcal{A}$ induces an injective, unital $*$-homomorphism $\alpha: \mathcal{A}\rightarrow \mathcal{A}\otimes C(G)$ in the following way:
\begin{gather}\label{eqn:induced_coaction_from_classical}
    \alpha(a)(g) := \sigma_{g}(a), \;\;\;\; g \in G, a \in \mathcal{A},
\end{gather}
\noindent under the identification $\mathcal{A}\otimes C(G) \cong C(G; \mathcal{A})$. For any $f \in L^{1}(G, \mu)$, we may induce a linear functional $\omega_{f}: C(G)\rightarrow \mathbb{C}$ by setting
\begin{gather}\label{eqn:induced_functional_from_classical}
    \omega_{f}(\psi) := \int\limits_{G}f(g)\psi(g)d\mu(g), \;\;\;\; \psi \in C(G).
\end{gather}
\noindent That this is a well-defined bounded linear functional is easily derived from the properties of $f \in L^{1}(G, \mu)$ and the fact that $\mu$ is a probability measure on $G$. By the previous argument and the Riesz-Markov-Kakutani representation theorem, we have an injective embedding of $L^{1}(G, \mu_{G}) \hookrightarrow C(G)^{*} \cong M(G)$, the space of bounded complex Radon measures on $G$.

If $f \in L^{1}(G, \mu_{G})$ and $a \in \mathcal{A}$, we define the operator $f \star a \in \mathcal{B}(L^{2}(G, \mu_{G}))$ as the strong integral
\begin{gather*}
    f \star a := \int\limits_{G}f(g)\sigma_{g}(a)d\mu_{G}(g). 
\end{gather*} 
\noindent Using a density argument, one can then show that for $f \in L^{1}(G, \mu_{G})$ and $a \in \mathcal{A}$,
\begin{gather*}
    f\star a 
    = 
    \int\limits_{G}f(g)\sigma_{g}(a)d\mu_{G}(g)
    =
    \int\limits_{G}f(g)\alpha(a)(g)d\mu_{G}(g)
    =
    ({\rm id}_{\mathcal{A}}\otimes \phi_{f})\alpha(a) 
    =
    \phi_{f}\star a. 
\end{gather*}
\noindent This shows how (\ref{eqn:convolution_for_algebra}) recovers function-operator convolution for a classical group acting on a ${\rm C}^{*}$-algebra, as indicated first in the approach of \cite{Werner84}.
\end{remark} 

\subsection{Lip-norms on compact quantum groups}
The following definitions and results were taken from \cite{Li09}. For a compact quantum group $\bb{G} = (C(\bb{G}), \Delta)$, a Lip-norm $L$ on $C(\bb{G})$ is called \textit{regular} if $L|_{\cl{O}(\bb{G})} < \infty$; that is, if $L$ acting only on the dense Hopf $*$-algebra corresponding to $\bb{G}$ is finite. A regular Lip-norm $L$ on $C(\bb{G})$ is said to be \textit{left-invariant} (respectively, \textit{right-invariant}) if for all $x \in C(\bb{G})_{h}$ and all $\mu \in S(C(\bb{G}))$ we have
\begin{gather*}
    L(x\ast \mu) \leq L(x) \;\;\;\; ({\rm respectively}, \; L(\mu\ast x) \leq L(x)).
\end{gather*}
\noindent We say $L$ is \textit{bi-invariant}, if it is both left and right-invariant. The results surrounding \cite[Proposition 8.9]{Li09} show that if $\mathbb{G}$ is separable (i.e., $C(\bb{G})$ is separable as a Banach space) and co-amenable, then there exist bi-invariant regular Lip-norms on $\bb{G}$. 

If $\mathbb{G} = (C(\bb{G}), \Delta)$ is a compact quantum group acting on unital ${\rm C}^{*}$-algebra $\cl{A}$ via $\alpha: \cl{A}\rightarrow \cl{A}\otimes C(\bb{G})$, where the action is ergodic and $\bb{G}$ is co-amenable, then any regular Lip-norm $L_{\bb{G}}$ on $C(\bb{G})$ induces a corresponding Lip-norm $L_{\cl{A}}$ on $\cl{A}$, defined by
\begin{gather}\label{eqn:induced_lip_norm_on_a}
    L_{\cl{A}}(a) := \sup\limits_{\phi \in S(\cl{A})}L_{\bb{G}}(a\star\phi),
\end{gather}
\noindent where $a \star \phi$ is as given in (\ref{eqn:convolution_for_algebra}), and that $L_{\cl{A}}$ is regular on $\cl{A}$. Furthermore, the induced Lip-norm is $\alpha$-invariant, in that
\begin{gather*}
    L_{\cl{A}}(\mu\star a) \leq L_{\cl{A}}(a),
\end{gather*}
\noindent for any $a \in \cl{A}_{h}$ and all $\mu \in S(C(\bb{G}))$ where $\mu\star a$ is as given in (\ref{eqn:convolution_for_algebra}).

\section{Approximations for compact quantum groups acting on ${\rm C}^{*}$-algebras}\label{sec:approximations_from_cqg_actions}

In all of what follows, we set $\bb{G} = (C(\bb{G}), \Delta)$ as a co-amenable, separable compact quantum group acting on a unital ${\rm C}^{*}$-algebra $\mathcal{A}$ via co-action $\alpha: \mathcal{A}\rightarrow \cl{A}\otimes C(\bb{G})$. We fix $L_{\mathbb{G}}$ to be any bi-invariant regular Lip-norm on $\mathbb{G}$; we then let $L_{\cl{A}}:\mathcal{A}\rightarrow [0, \infty]$ denote the induced Lip-norm on $\mathcal{A}$ as given in (\ref{eqn:induced_lip_norm_on_a}). Now, and in the sequel, assume that $\{\varphi_{t}\}_{t > 0} \subseteq L^{1}(\mathbb{G})$ is a bounded approximate identity. 
While the following lemma is almost certainly well-known to the experts, we provide a proof for the convenience of the reader. 

\begin{lemma}\label{lem:norm_density_of_convolution_in_a}
Let $\mathcal{A}$ be a unital ${\rm C}^{*}$-algebra, and let $\mathbb{G} = (C(\mathbb{G}), \Delta)$ be a co-amenable compact quantum group with co-action $\alpha: \mathcal{A}\rightarrow \mathcal{A}\otimes C(\mathbb{G})$. Then $L^{1}(\mathbb{G})\star \mathcal{A}$ is norm-dense in $\mathcal{A}$. 
\end{lemma}
\begin{proof}
Start by fixing an irreducible representation $\pi \in {\rm Irr}(\bb{G})$, and let $\{e_{i}^{\pi}\}_{i=1}^{n_{\pi}}$ be an orthonormal basis for $H_{\pi}$ corresponding to $\pi$. Set 
\begin{gather*}
    \chi_{\pi} := \sum\limits_{i=1}^{n_{\pi}}{\rm Tr}(Q_{\pi})U_{\pi}(e_{i}, Q_{\pi}^{-1}e_{i}), 
\end{gather*}
\noindent where $U_{\pi}(e_{i}, Q_{\pi}^{-1}e_{i})$ is the matrix coefficient in $C(\bb{G})$ defined in the discussion surrounding (\ref{eqn:comult_on_matrix_coeff}). For $a \in \mathcal{A}$, define $E_{\pi}(a) := ({\rm id}_{\cl{A}}\otimes \varphi_{\bb{G}})(\alpha(a)(1_{\cl{A}}\otimes \chi_{\pi}^{*}))$. Using properties of slice maps, we have that
\begin{eqnarray*}
    E_{\pi}(a) 
    & = &
    ({\rm id}_{\cl{A}}\otimes \varphi_{\bb{G}})(\alpha(a)(1_{\cl{A}}\otimes \chi_{\pi}^{*})) \\
    & = &
    ({\rm id}_{\cl{A}}\otimes (\chi_{\pi}^{*}\cdot \varphi_{\bb{G}}))\alpha(a) \\
    & = &
    (\chi_{\pi}^{*}\cdot \varphi_{\bb{G}})\star a,
\end{eqnarray*}
\noindent where for $\varphi \in C(\bb{G})^{*}$ and $x \in C(\bb{G})$ we let $x\cdot \varphi \in C(\bb{G})^{*}$ be the functional defined via $(x\cdot \varphi)(y) := \varphi(yx)$, for $y \in C(\bb{G})$. For a fixed representation, we set $\cl{A}^{\pi} \subseteq \cl{A}$ as the subspace consisting of all $a \in \cl{A}$ such that $E_{\pi}(a) = a$. The algebraic core $\cl{O}_{\bb{G}}(\cl{A})$ of $\bb{G}$ acting on $\cl{A}$ is defined as 
\begin{gather*}
    \cl{O}_{\bb{G}}(\cl{A}) = \bigoplus\limits_{\pi \in {\rm Irr}(\bb{G})}\cl{A}^{\pi},
\end{gather*}
\noindent and by \cite[Theorem 1.5]{Podles95} this is a norm-dense, unital $*$-subalgebra of $\cl{A}$. Now, as $\bb{G}$ is a compact quantum group, $C(\bb{G})$ is unital, with $\Lambda(1_{C(\bb{G})}) = \xi_{\bb{G}} \in L^{2}(\bb{G})$ from the GNS representation using $\varphi_{\bb{G}}$. By construction, we also see for each $\pi \in {\rm Irr}(\bb{G})$ that
\begin{gather*}
    (\chi_{\pi}^{*}\cdot \varphi_{\bb{G}})(x) = \varphi_{\bb{G}}(x\chi_{\pi}^{*}) = \langle \xi_{\bb{G}}, \pi_{\bb{G}}(x\chi_{\pi}^{*})\xi_{\bb{G}}\rangle = \langle \pi_{\bb{G}}(\chi_{\pi})\xi_{\bb{G}}, \pi_{\bb{G}}(x)\xi_{\bb{G}}\rangle \\= \langle \Lambda(\chi_{\pi}), \pi_{\bb{G}}(x)\Lambda(1_{C(\bb{G})})\rangle = \omega_{\Lambda(\chi_{\pi}), \Lambda(1_{C(\bb{G})})}(x),
\end{gather*}
\noindent showing $\chi_{\pi}^{*}\cdot \varphi_{\bb{G}}$ can be written as a vector functional on $C(\bb{G})$. This extends to a normal functional on $L^{\infty}(\bb{G})$ (which we denote in the same way), and thus may be considered as an element of $L^{1}(\bb{G})$. Therefore, for every $\pi \in {\rm Irr}(\bb{G})$, we have that $\chi_{\pi}^{*}\cdot \varphi_{\bb{G}} \in L^{1}(\bb{G})$, and so $(\chi_{\pi}^{*}\cdot \varphi_{\bb{G}})\star a \in L^{1}(\bb{G})\star \cl{A}$ for all $a \in \cl{A}$. This implies $\cl{O}_{\bb{G}}(\cl{A}) \subseteq L^{1}(\bb{G})\star \cl{A}$; as the former space is norm-dense inside $\cl{A}$, this also shows $L^{1}(\bb{G})\star \cl{A}$ must be norm-dense inside $\cl{A}$ as claimed. 
\end{proof}

\begin{proposition}\label{prop:properties_of_convolution}
Let $\mathcal{A}$ be a unital ${\rm C}^{*}$-algebra, and let $\mathbb{G} = (C(\mathbb{G}), \Delta)$ be a co-amenable compact quantum group with co-action $\alpha: \mathcal{A}\rightarrow \mathcal{A}\otimes C(\mathbb{G})$. 
\begin{itemize}
    \item[(i)] If $\omega \in L^{1}(\mathbb{G})$ is a positive functional, and $a \in \mathcal{A}_{+}$, then $\omega \star a \in \mathcal{A}_{+}$. 
    \item[(ii)] Suppose $\{\varphi_{t}\}_{t} \subseteq L^{1}(\mathbb{G})$ is a bounded approximate identity, and $a \in \mathcal{A}$. Then $\varphi_{t}\star a\rightarrow a$ in the ${\rm C}^{*}$-norm of $\mathcal{A}$ as $t\rightarrow \infty$. 
\end{itemize}
\end{proposition}
\begin{proof}
(i) Assume $\omega \in L^{1}(\mathbb{G})$ is positive--- that is, $\omega(x) \geq 0$ for all $x \in C(\mathbb{G})^{+}$. As $C(\mathbb{G})$ is a ${\rm C}^{*}$-algebra, and as $\omega: C(\mathbb{G})\rightarrow \mathbb{C}$ is a positive linear functional with commutative range, then $\omega$ is automatically completely positive (see \cite[Proposition 3.8]{paulsen}). As $\alpha$ is a $*$-homomorphism between ${\rm C}^{*}$-algebras, it is also completely positive (as each amplification $\alpha^{(n)}$ for $n \in \mathbb{N}$ is positive and contractive as a $*$-homomorphism on the matricial level); furthermore, it is trivial that ${\rm id}_{\mathcal{A}}: \mathcal{A}\rightarrow \mathcal{A}$ acts as a completely bounded, completely positive map on $\mathcal{A}$. Thus, $({\rm id}_{\mathcal{A}} \otimes \omega): \mathcal{A}\otimes C(\mathbb{G})\rightarrow \mathcal{A}$ acts as a completely positive map between ${\rm C}^{*}$-algebras (see, for instance, \cite[Corollary IV.4.25]{Takesaki03}), with $({\rm id}_{\mathcal{A}}\otimes \omega) \alpha: \mathcal{A}\rightarrow \mathcal{A}$ a completely positive map (as the composition of two completely positive maps). Therefore, if $a \in \mathcal{A}_{+}$, $\omega \star a = ({\rm id}_{\mathcal{A}}\otimes \omega)\alpha(a) \in \mathcal{A}_{+}$ as claimed.

(ii) Let $a \in \mathcal{A}$ and $\{\varphi_{t}\}_{t} \subseteq L^{1}(\mathbb{G})$ be as above. Take $\epsilon > 0$ arbitrary but fixed. For the moment, take arbitrary $b = ({\rm id}_{\mathcal{A}}\otimes \omega)\alpha(\hat{a}) \in L^{1}(\mathbb{G})\star \mathcal{A}$, and note that (using the compatibility relations arising from the co-action) for any $\varphi_{t}$, 
\begin{eqnarray*}
    \varphi_{t}\star b 
    & = &
    ({\rm id}_{\mathcal{A}}\otimes \varphi_{t})\alpha(b) \\
    & = &
    ({\rm id}_{\mathcal{A}}\otimes \varphi_{t})\alpha(({\rm id}_{\mathcal{A}}\otimes \omega)\alpha(\hat{a})) \\
    & = &
    ({\rm id}_{\mathcal{A}}\otimes \varphi_{t}\otimes \omega)(\alpha\otimes {\rm id}_{C(\mathbb{G})})\alpha(\hat{a}) \\
    & = &
    ({\rm id}_{\mathcal{A}}\otimes \varphi_{t}\otimes \omega)({\rm id}_{\mathcal{A}}\otimes \Delta)\alpha(\hat{a}) \\
    & = &
    ({\rm id}_{\mathcal{A}}\otimes \varphi_{t}\star \omega)\alpha(\hat{a}) \\
    & = &
    (\varphi_{t}\star\omega)\star \hat{a}.
\end{eqnarray*}
\noindent As $\{\varphi_{t}\}_{t}$ is a bounded approximate identity, there exists $T > 0$ such that if $t \geq T$, $\|(\varphi_{t}\star \omega)-\omega\|_{L^{1}(\mathbb{G})} \leq \epsilon/2\|\alpha(\hat{a})\|$, and thus for $t \geq T$ we have that $\|\varphi_{t}\star b-b\| \leq \epsilon/2$. Now, using density of $L^{1}(\mathbb{G})\star \mathcal{A}$ inside $\mathcal{A}$ (via Lemma \ref{lem:norm_density_of_convolution_in_a}), pick $b \in L^{1}(\mathbb{G})\star \mathcal{A}$ such that $\|a-b\| < \epsilon/4$. For the same $t \geq T$ (dependent on our choice of $b$) as before, we have that
\begin{eqnarray*}
    \|\varphi_{t}\star a-a\| 
    & \leq &
    \|\varphi_{t}\star a-\varphi_{t}\star b\|+\|\varphi_{t}\star b-b\|+\|b-a\| \\
    & \leq &
    \|a-b\|+\|\varphi_{t}\star b-b\|+\|b-a\| \\
    & \leq &
    2\bigg(\frac{\epsilon}{4}\bigg)+\|\varphi_{t}\star b-b\| \\
    & \leq &
    \frac{\epsilon}{2}+\frac{\epsilon}{2} = \epsilon,
\end{eqnarray*}
\noindent where we also use the fact that $\|\omega \star a\| \leq \|\omega\|_{L^{1}(\mathbb{G})}\|a\|$. This implies the claim of (iii). 
\end{proof}

\begin{corollary}\label{cor:norm_convergence_for_bai}
Under previous assumptions, we have that $\varphi_{t}\star x\rightarrow x$ in the ${\rm C}^{*}$-norm of $C(\mathbb{G})$ for all $x \in C(\mathbb{G})$. 
\end{corollary}
\begin{proof}
This follows from Proposition \ref{prop:properties_of_convolution} (ii), when we consider $\Delta: C(\mathbb{G})\rightarrow C(\mathbb{G})\otimes C(\mathbb{G})$ as a co-action of $\mathbb{G}$ on itself. 
\end{proof}

Under previous assumptions, set
\begin{gather}\label{eqn:set_k}
    K := \{x \in C(\mathbb{G}): \; x \in B_{1}^{L_{\bb{G}}}(0),  \; \varepsilon(x) = 0\}
\end{gather}
\noindent as a subset of $C(\mathbb{G})$. Define the function $\|\cdot\|_{*}: C(\mathbb{G})^{*}\rightarrow [0, \infty)$ via
\begin{gather}\label{eqn:seminorm}
    \|\eta\|_{*} := \sup\{|\eta(x)|: \; x \in K\}.
\end{gather}
\begin{lemma}\label{lem:seminorm_convergence}
The following hold:
\begin{itemize}
    \item[(i)] $\|\cdot\|_{\ast}$ is a semi-norm on $C(\mathbb{G})^{*}$;
    \item[(ii)] If $\{\varphi_{t}\}_{t > 0} \subseteq L^{1}(\mathbb{G})$ is an $L^{1}(\mathbb{G})$-approximate identity, we have that $\|\varphi_{t}-\epsilon\|_{\ast}\rightarrow 0$ as $t \rightarrow \infty$.
\end{itemize}
\end{lemma}
\begin{proof}
(i) To see that $\|\cdot\|_{\ast}$ is well-defined (i.e., that it is finite), first note that as $L_{\mathbb{G}}$ is a Lip-norm on $C(\mathbb{G})$, by \cite[Theorem 1.8]{rieffel} there exists some $M > 0$ such that $\inf\limits_{\lambda \in \mathbb{C}}\|x-\lambda1_{C(\mathbb{G})}\| \leq M$ for $x \in C(\mathbb{G})$ with $L_{\mathbb{G}}(x) \leq 1$. For $\epsilon > 0$ and $x \in K$, pick $\lambda_{\epsilon} \in \mathbb{C}$ such that $\|x-\lambda_{\epsilon}1_{C(\mathbb{G})}\| \leq M+\epsilon$. As $\varepsilon$ is multiplicative on $C(\mathbb{G})$, then $\varepsilon(1_{C(\mathbb{G})}) = 1$ and so
\begin{gather*}
    0 = \varepsilon(x) = \varepsilon(x-\lambda_{\epsilon}1_{C(\mathbb{G})})+\varepsilon(\lambda_{\epsilon}1_{C(\mathbb{G})}) = \varepsilon(x-\lambda_{\epsilon}
    0 = \varepsilon(x) = \varepsilon(x-\lambda_{\epsilon}1_{C(\mathbb{G})})+\lambda_{\epsilon},
\end{gather*}
\noindent and so $\lambda_{\epsilon} = -\varepsilon(x-\lambda_{\epsilon}1_{C(\mathbb{G})})$. Additionally, by co-amenability as $\varepsilon$ is norm bounded we have
\begin{gather*}
    |\lambda_{\epsilon}| = |\varepsilon(x-\lambda_{\epsilon}1_{C(\mathbb{G})})| \leq \|\varepsilon\|\|x-\lambda_{\epsilon}1_{C(\mathbb{G})}\|.
\end{gather*}
\noindent Thus, 
\begin{gather*}
    \|x\| \leq \|x-\lambda_{\epsilon}1_{C(\mathbb{G})}\|+|\lambda_{\epsilon}| \leq (1+\|\varepsilon\|)\|x-\lambda_{\epsilon}1_{C(\mathbb{G})}\| \leq (1+\|\varepsilon\|)(M+\epsilon).
\end{gather*}
\noindent As our choice of $\epsilon > 0$ was arbitrary, this shows that the set $K$ is norm bounded in $C(\mathbb{G})$. Thus, $\|\cdot\|_{\ast}$ is finite. To see why it is a semi-norm, let $\eta, \omega \in C(\mathbb{G})^{*}$ and $\lambda \in \mathbb{C}$. Then
\begin{eqnarray*}
    \|\eta+\lambda\omega\|_{\ast} 
    & = &
    \sup\{|\eta(x)+\lambda\omega(x)|: \; x \in K\} \\
    & \leq &
    \sup\{|\eta(x)|+|\lambda||\omega(x)|: \; x \in K\} \\
    & \leq & 
    \sup\{|\eta(x)|: \; x \in K\}+|\lambda|\sup\{|\omega(y)|: \; y \in K\} \\
    & = &
    \|\eta\|_{\ast}+|\lambda|\|\omega\|_{\ast},
\end{eqnarray*}
\noindent establishing our claim. 

(ii) Recall that, as $\mathbb{G}$ is co-amenable we have an injective embedding $L^{1}(\mathbb{G}) \hookrightarrow C(\mathbb{G})^{*}$; this allows us to compute $\|\varphi_{t}\|_{\ast}$ considered inside $C(\mathbb{G})^{*}$. To establish (ii), we first show that $K$ is totally norm bounded in $C(\mathbb{G})$. To that end, note that in the proof of (i) we show that if $q: C(\mathbb{G})\rightarrow C(\mathbb{G})/\mathbb{C}$ denotes the canonical quotient map from $C(\mathbb{G})$ to the quotient Banach space $C(\mathbb{G})/\mathbb{C}$, then $\|x\| \leq (1+\|\varepsilon\|)\|q(x)\|$ for $x \in K$. Furthermore, if $y \in K$, then $\varepsilon(x-y) = \varepsilon(x)-\varepsilon(y) = 0$ and so $\|x-y\| \leq (1+\|\varepsilon\|)\|q(x-y)\|$ as well. As mentioned, by \cite[Theorem 1.8]{rieffel} the set $q(B_{1}^{L_{\mathbb{G}}}(0))$ is totally bounded in $C(\mathbb{G})/\mathbb{C}$ under the quotient norm. Since $K \subseteq B_{1}^{L_{\bb{G}}}(0)$, $q(K)$ is also totally bounded in $C(\mathbb{G})/\mathbb{C}$. Thus, for $\delta > 0$ pick $x_{1}, \hdots, x_{n} \in K$ such that
\begin{gather*}
    q(K) \subseteq \bigcup\limits_{j=1}^{n}\mathcal{B}_{\bigg(\frac{\delta}{1+\|\varepsilon\|}\bigg)}^{C(\mathbb{G})/\mathbb{C}}\bigg(q(x_{j})\bigg).
\end{gather*}
\noindent For any $x \in K$, there exists some $j \in [n]$ such that $\|q(x-x_{j})\| \leq \frac{\delta}{1+\|\varepsilon\|}$. Then by our previous inequality, we have
\begin{gather*}
    \|x-x_{j}\|\leq (1+\|\varepsilon\|)\|q(x-x_{j})\| \leq \delta.
\end{gather*}
\noindent As our choice of $x \in K$ was arbitrary, this implies $K$ is totally norm bounded as claimed. 

Now, using the identity $(\varepsilon \otimes {\rm id}_{C(\mathbb{G})})\Delta = {\rm id}_{C(\mathbb{G})}$, we have
\begin{gather*}
    \varphi_{t}(x) = \varphi_{t}({\rm id}_{C(\mathbb{G})}(x)) = \varphi_{t}\circ (\varepsilon\otimes {\rm id}_{C(\mathbb{G})})\Delta(x) = (\varepsilon\otimes \varphi_{t})\Delta(x) = \varepsilon(\varphi_{t}\star x).
\end{gather*}
\noindent By Corollary \ref{cor:norm_convergence_for_bai} we have that $\varphi_{t}\star x \rightarrow x$ in norm for any $x \in C(\mathbb{G})$. As $\varepsilon$ is norm bounded, this yields $\varphi_{t}(x) = \varepsilon(\varphi_{t}\star x)\rightarrow \varepsilon(x)$ for any $x \in C(\mathbb{G})$ as $t\rightarrow \infty$. Thus, $(\varphi_{t}-\varepsilon)(x) \rightarrow 0$ as $t\rightarrow \infty$. As $\{\varphi_{t}\}_{t > 0}$ is an $L^{1}(\mathbb{G})$-approximate identity, $\sup\limits_{t}\|\varphi_{t}\|_{L^{1}(\mathbb{G})} = 1$; let $\hat{M} := 1+\|\varepsilon\|$. Fix arbitrary $\epsilon > 0$, set $\delta = \epsilon/2\hat{M}$, and pick elements $x_{1}, \hdots, x_{n} \in K$ corresponding to the finite $\delta$-net (which we can do, as we have shown $K$ is totally norm bounded).

For each $j \in [n]$, we have that $(\varphi_{t}-\varepsilon)(x_{j})\rightarrow 0$ as $t\rightarrow \infty$. This means (as we have a finite number of $x_{j}$'s for our $\delta$-net) there exists some $t_{0} \geq 0$ such that if $t \geq t_{0}$, then $|(\varphi_{t}-\varepsilon)(x_{j})| < \epsilon/2$ for all $j \in [n]$. Pick any $x \in K$; we have some $j \in [n]$ such that $\|x-x_{j}\| < \delta$. Now, 
\begin{eqnarray*}
    |(\varphi_{t}-\varepsilon)(x)| 
    & \leq &
    |(\varphi_{t}-\varepsilon)(x-x_{j})|+|(\varphi_{t}-\varepsilon)(x_{j})| \\
    & \leq &
    \hat{M}\|x-x_{j}\|+\frac{\epsilon}{2} \\
    & < &
    \hat{M}\bigg(\frac{\epsilon}{2\hat{M}}\bigg)+\frac{\epsilon}{2} \\
    & = & 
    \epsilon,
\end{eqnarray*}
\noindent for $t \geq t_{0}$. This implies $\sup\limits_{x \in K}|(\varphi_{t}-\varepsilon)(x)|\rightarrow 0$ as $t\rightarrow \infty$. However, 
\begin{gather*}
    \sup\limits_{x \in K}|(\varphi_{t}-\varepsilon)(x)| = \sup\{|(\varphi_{t}-\varepsilon)(x)|: \; x \in K\} = \|(\varphi_{t}-\varepsilon)\|_{\ast}.
\end{gather*}
\noindent Thus, $\|(\varphi_{t}-\varepsilon)\|_{\ast}\rightarrow 0$ as $t\rightarrow \infty$, as claimed. 
\end{proof}

Now, and in the remainder of this section, assume that $\{\varphi_{t}\}_{t > 0}$ is not just a bounded approximate identity but an $L^{1}(\mathbb{G})$-approximate identity. For fixed $t > 0$, set 
\begin{gather*}
    \mathcal{A}_{t} := \{\varphi_{t}\star a: \; a \in \mathcal{A}\},
\end{gather*}
\noindent considered as the vector space image of $\mathcal{A}$ under convolution by some fixed $\varphi_{t}$. Additionally, define $L_{\cl{A}_{t}}: \mathcal{A}_{t}\rightarrow [0, \infty]$ by
\begin{gather*}
    L_{\mathcal{A}_{t}}(\hat{a}) := L_{\mathcal{A}}(\hat{a}), \;\;\;\; \hat{a} \in \mathcal{A}_{t}.
\end{gather*}
\begin{proposition}\label{prop:subspace_is_operator_system}
The space $\mathcal{A}_{t}$ is an operator system in $\mathcal{B}(L^{2}(\mathbb{G}))$. 
\end{proposition}
\begin{proof}
The proof of Proposition \ref{prop:properties_of_convolution} (i) shows that for a fixed state $\varphi_{t} \in L^{1}(\mathbb{G})$, the mapping $a \mapsto \varphi_{t}\star a$ is completely positive. Furthermore, as $\varphi_{t}$ is a state, and $\alpha$ is unital (by definition), we have that 
\begin{gather*}
    \varphi_{t}\star 1_{\mathcal{A}} = ({\rm id}_{\mathcal{A}}\otimes \varphi_{t})\alpha(1_{\mathcal{A}}) = ({\rm id}_{\mathcal{A}}\otimes \varphi_{t})(1_{\mathcal{A}}\otimes 1_{C(\mathbb{G}}) = \varphi_{t}(1_{C(\mathbb{G})})1_{\mathcal{A}} = 1_{\mathcal{A}},
\end{gather*}
\noindent and thus the mapping
\begin{gather*}
    \mathcal{A} \rightarrow \mathcal{A}, \;\;\;\; a \mapsto \varphi_{t}\star a,
\end{gather*}
\noindent is a unital, completely positive map. As $\mathcal{A}$ is a ${\rm C}^{*}$-algebra, it is trivially an operator system; thus, $\mathcal{A}_{t}$ is the image of an operator system under a UCP map, and hence an operator system as claimed. 
\end{proof}

\begin{proposition}\label{prop:spaces_are_cqms}
For each $t > 0$, pairs $(\mathcal{A}, L_{\mathcal{A}})$ and $(\mathcal{A}_{t}, L_{\mathcal{A}_{t}})$ are compact quantum metric spaces. 
\end{proposition}
\begin{proof}
That $(\mathcal{A}, L_{\mathcal{A}})$ is a compact quantum metric space follows from the results of \cite{Li09}, as discussed in the comments surrounding (\ref{eqn:induced_lip_norm_on_a}). To see why $(\cl{A}_{t}, L_{\cl{A}_{t}})$ is a compact quantum metric space, note that by Proposition \ref{prop:subspace_is_operator_system} the subspace $\mathcal{A}_{t}$ is an operator system. As ${\rm dom}(L_{\mathcal{A}})$ is dense in $\mathcal{A}$, with $\mathcal{A}_{t} \subseteq \mathcal{A}$, then ${\rm dom}(L_{\mathcal{A}_{t}}) = {\rm dom}(L_{\mathcal{A}}) \cap \mathcal{A}_{t}$ is dense in $\mathcal{A}_{t}$. As $\mathcal{A}_{t}$ shares the same unit as $\mathcal{A}$, by properties of $L_{\mathcal{A}}$ as a Lip-norm we have that $L_{\mathcal{A}_{t}}(\hat{a}^{*})= L_{\mathcal{A}_{t}}(\hat{a})$ for every $\hat{a} \in \mathcal{A}_{t}$ and $\mathbb{C}1_{\mathcal{A}} = \ker(L_{\mathcal{A}_{t}})$. Thus, it suffices to show that induced metric $d^{L_{\mathcal{A}_{t}}}$ metrizes the ${\rm weak}^{*}$-topology on $S(\mathcal{A}_{t})$. As
\begin{gather*}
    B_{1}^{L_{\cl{A}_{t}}}(0) = \{\hat{a} \in \mathcal{A}_{t}: \; L_{\mathcal{A}_{t}}(\hat{a}) \leq 1\} \subseteq \{a \in \mathcal{A}: \; L_{\mathcal{A}}(a) \leq 1\} = B_{1}^{L_{\cl{A}}}(0),
\end{gather*}
\noindent and the latter is totally norm bounded in $\mathcal{A}/\mathbb{C}$, then the former is as well. By \cite[Theorem 1.8]{rieffel}, this shows $d^{L_{\mathcal{A}_{t}}}$ metrizes the ${\rm weak}^{*}$-topology on $S(\mathcal{A}_{t})$ as desired. Thus, $(\mathcal{A}_{t}, L_{\mathcal{A}_{t}})$ is a compact quantum metric space by definition.
\end{proof}

Define the linear maps $\Phi_{t} : \mathcal{A}\rightarrow \mathcal{A}_{t}$ and $\Psi_{t}: \mathcal{A}_{t}\rightarrow \mathcal{A}$ as follows:
\begin{gather*}
    \Phi_{t}(a) := \varphi_{t}\star a, \;\;\;\; a \in \mathcal{A},
    \;\;\;\;\;\;\;\; \Psi_{t}(\hat{a}) := \iota(\hat{a}), \;\;\;\; \hat{a} \in \mathcal{A}_{t},
\end{gather*}
\noindent where $\mathcal{A}_{t} \hookrightarrow \mathcal{A}$ is the completely isometric inclusion of $\mathcal{A}_{t}$ inside $\mathcal{A}$. 

\begin{proposition}\label{prop:existence_of_gammas}
For $t > 0$ and $a \in \mathcal{A}$, we have
\begin{gather*}
    \|\Psi_{t}\circ \Phi_{t}(a)-a\| \leq \gamma_{t}L_{A}(a),
\end{gather*}
\noindent where $\gamma_{t} > 0$ such that $\gamma_{t}\rightarrow 0$ as $t \rightarrow \infty$. 
\end{proposition}
\begin{proof}
For ease of notation, let $1$ denote the unit of $C(\mathbb{G})$. Note that for $t >0$, by definition $\Psi_{t}\circ \Phi_{t}(a) = \varphi_{t}\star a = ({\rm id}_{\cl{A}}\otimes \varphi_{t})\alpha(a)$. As $\mathbb{G}$ acts on $\mathcal{A}$ ergodically and $\mathbb{G}$ is co-amenable, $\varepsilon: C(\mathbb{G})\rightarrow \mathbb{C}$ is bounded with $\varepsilon\star a = ({\rm id}_{\cl{A}}\otimes \varepsilon)\alpha(a) = a$ (by Remark \ref{rem:co-unit_identity}). Then
\begin{gather*}
    \varphi_{t}\star a-a = \varphi_{t}\star a-\varepsilon\star a = (\varphi_{t}-\varepsilon)\star a = ({\rm id}_{\cl{A}}\otimes (\varphi_{t}-\varepsilon))\alpha(a).
\end{gather*}
\noindent Set $\eta_{t} := \varphi_{t}-\varepsilon$. We claim that for $x \in C(\mathbb{G})$, $|\eta_{t}(x)| \leq \|\eta_{t}\|_{\ast}L_{\mathbb{G}}(x)$. Define $K$ as in (\ref{eqn:set_k}). We note that as $\varphi_{t}(1) = 1 = \varepsilon(1)$, then $\eta_{t}(1) = 0$. If $x \in C(\mathbb{G})$ with $\varepsilon(x) \neq 0$, set $\hat{x} := x-\varepsilon(x)1$. Then 
\begin{gather*}
    \varepsilon(\hat{x}) = \varepsilon(x-\varepsilon(x)1) = \varepsilon(x)-\varepsilon(x) = 0, \\\eta_{t}(\hat{x}) = \eta_{t}(x-\varepsilon(x)1)= \eta_{t}(x)-\varepsilon(x)\eta_{t}(1) = \eta_{t}(x),
\end{gather*}
\noindent and as $\mathbb{C}1 = \ker(L_{\mathbb{G}})$ (by definition of a Lip-norm on $\mathbb{G}$) we have
\begin{gather*}
    L_{\mathbb{G}}(\hat{x}) = L_{\mathbb{G}}(x-\varepsilon(x)1) = L_{\mathbb{G}}(x)+L_{\mathbb{G}}(\varepsilon(x)1) = L_{\mathbb{G}}(x).
\end{gather*}
\noindent Thus, without loss of generality we assume we are working with $x \in C(\mathbb{G})$ such that $\varepsilon(x) = 0$. If $L_{\mathbb{G}}(x) = 0$, then the claimed inequality is trivial as $x = \lambda 1$ for some $\lambda \in \mathbb{C}$, with $\eta_{t}(x) = 0$. If $L_{\mathbb{G}}(x) > 0$, set $y: = x/L_{\mathbb{G}}(x)$ so that $\varepsilon(y) = 0$ with $L_{\mathbb{G}}(y) = 1$. Then $y \in K$, and so $|\eta_{t}(y)| \leq \|\eta_{t}\|_{\ast}$; thus, $|\eta_{t}(x)| \leq \|\eta_{t}\|_{\ast}L_{\mathbb{G}}(x)$. Therefore, $|\eta_{t}(x)| \leq \|\eta_{t}\|_{\ast}L_{\mathbb{G}}(x)$ for all $x \in C(\mathbb{G})$ as claimed. 

For $\phi \in S(\cl{A})$, by properties of slice maps (see our comments in Section \ref{ss:hopf_alg_cqg}) we have
\begin{eqnarray*}
    \phi(\varphi_{t}\star a-a) 
    & = &
    \phi(\eta_{t}\star a) \\
    & = &
    \phi(({\rm id}_{\cl{A}}\otimes \eta_{t})\alpha(a)) \\
    & = &
    \eta_{t}((\phi\otimes {\rm id}_{C(\mathbb{G})})\alpha(a)) \\
    & = &
    \eta_{t}(a\star \phi).
\end{eqnarray*}

Now, assume for the moment that $a \in \cl{A}_{h}$. As $\Psi_{t}, \Phi_{t}$ are UCP maps (as shown in the proof of Proposition \ref{prop:subspace_is_operator_system} and surrounding comments), the element $\Psi_{t}\circ \Phi_{t}(a)$ is the image of a self-adjoint element under a UCP map and hence must be self-adjoint. Therefore, $\Psi_{t}\circ\Phi_{t}(a)-a$ is self-adjoint in $\cl{A}$, and so by our previous equality 
\begin{gather*}
    \|\Psi_{t}\circ\Phi_{t}(a)-a\| = \|\eta_{t}\star a\| = \sup\limits_{\phi \in S(\cl{A})}|\phi(\eta_{t}\star a)| = \sup\limits_{\phi \in S(\cl{A})}|\eta_{t}(a\star \phi)|.
\end{gather*}
\noindent As $a \star \phi \in C(\mathbb{G})$ for each $\phi \in S(\cl{A})$, we have that $|\eta_{t}(a\star \phi)| \leq \|\eta_{t}\|_{\ast}L_{\mathbb{G}}(a\star \phi)$ for every $\phi \in S(\cl{A})$ and $t > 0$. By definition of the induced Lip-norm on $\cl{A}$, we then see
\begin{gather*}
    \sup\limits_{\phi \in S(\cl{A})}|\eta_{t}(a\star \phi)| \leq \|\eta_{t}\|_{\ast}\bigg(\sup\limits_{\phi \in S(\cl{A})}L_{\mathbb{G}}(a\star\phi)\bigg) = \|\eta_{t}\|_{\ast}L_{A}(a).
\end{gather*}
\noindent Now, for arbitrary (non self-adjoint) $a \in \cl{A}$, by Remark \ref{rem:aou_vs_op_system} we have
\begin{gather*}
    L_{\cl{A}}(a) = \sup\limits_{\theta \in [0, 2\pi]}L_{\cl{A}}(\cos(\theta) {\rm Re}(a)+\sin(\theta) {\rm Im}(a)),
\end{gather*}
\noindent where ${\rm Re}(a), {\rm Im}(a) \in \cl{A}_{h}$. Setting $\theta = 0$ or $\theta = \pi/2$ shows that $L_{\cl{A}}({\rm Re}(a)), L_{\cl{A}}({\rm Im}(a)) \leq L_{\cl{A}}(a)$. Thus, 
\begin{eqnarray*}
    \|\Psi_{t}\circ \Phi_{t}(a)-a\|
    & = &
    \|(\Psi_{t}\circ\Phi_{t}({\rm Re}(a))-{\rm Re}(a))+i(\Psi_{t}\circ\Phi_{t}({\rm Im}(a))-{\rm Im}(a))\| \\
    & \leq &
    \|\Psi_{t}\circ\Phi_{t}({\rm Re}(a))-{\rm Re}(a)\|+\|\Psi_{t}\circ\Phi_{t}({\rm Im}(a))-{\rm Im}(a)\| \\
    & \leq &
    \|\eta_{t}\|_{\ast}L_{\cl{A}}({\rm Re}(a))+\|\eta_{t}\|_{\ast}L_{\cl{A}}({\rm Im}(a)) \\
    & \leq &
    2\|\eta_{t}\|_{\ast}L_{\cl{A}}(a).
\end{eqnarray*}
\noindent Setting $\gamma_{t} := 2\|\eta_{t}\|_{\ast}$, an application of Lemma \ref{lem:seminorm_convergence} concludes the proof. 

\end{proof}

\begin{proposition}\label{prop:existence_of_gamma_primes}
For $t > 0$ and $\hat{a} \in \mathcal{A}_{t}$, we have
\begin{gather*}
    \|\Phi_{t}\circ \Psi_{t}(\hat{a})-\hat{a}\| \leq \gamma_{t}' L_{\cl{A}_{t}}(\hat{a}),
\end{gather*}
\noindent where $\gamma_{t}' > 0$ such that $\gamma_{t}' \rightarrow 0$ as $t\rightarrow \infty$. 
\end{proposition}
\begin{proof}
By definition of $\Psi_{t}$, we see that for any $\hat{a} \in \cl{A}_{t}$, then $\Phi_{t}\circ \Psi_{t}(\hat{a}) = \varphi_{t} \star \hat{a}$. As $L_{\cl{A}_{t}}(\hat{a}) = L_{\cl{A}}(\hat{a})$, the proof of Proposition \ref{prop:existence_of_gammas} shows that
\begin{gather*}
    \|\Phi_{t}\circ\Psi_{t}(\hat{a})-\hat{a}\| = \|\varphi_{t}\star \hat{a}-\hat{a}\| \leq \gamma_{t}L_{\cl{A}}(\hat{a}) = \gamma_{t}L_{\cl{A}_{t}}(\hat{a}).
\end{gather*}
\noindent Setting $\gamma_{t}' := \gamma_{t}$ concludes the proof. 
\end{proof}

\begin{theorem}\label{thm_approx_order_iso}
Let $I = (t)_{t \in I} \subseteq \mathbb{R}$ be a net of positive elements tending to infinity. Then $(\Phi_{t}, \Psi_{t})_{t \in I}$ form a ${\rm C}^{1}$-approximate order isomorphism for $\mathcal{A}$ and $\mathcal{A}_{t}$. 
\end{theorem}
\begin{proof}
We have already established that $\Phi_{t}, \Psi_{t}$ are UCP maps for each $t \in I$. It is clear that $\Psi_{t}$ is ${\rm C}^{1}$-contractive, as the inclusion map of $\mathcal{A}_{t}$ inside $\mathcal{A}$; that $\Phi_{t}$ is ${\rm C}^{1}$-contractive follows from the inequality $\|\varphi_{t}\star a\| \leq \|\varphi_{t}\|_{L^{1}(\mathbb{G})}\cdot\|a\| = \|a\|$ by our previous assumptions on each $\varphi_{t}$, for $t \in I$. An application of Propositions \ref{prop:existence_of_gammas} and \ref{prop:existence_of_gamma_primes} finishes the proof.
\end{proof}

\begin{corollary}\label{mainconv}
Compact quantum metric spaces $(\mathcal{A}_{t}, L_{\mathcal{A}_{t}})$ converge to $(\mathcal{A}, L_{\mathcal{A}})$ in the both the complete and the quantum Gromov-Hausdorff distance as $t\rightarrow \infty$. Furthermore, compact metric spaces $(S(\mathcal{A}_{t}), d^{L_{\mathcal{A}_{t}}})$ converge to $(S(\mathcal{A}), d^{L_{\mathcal{A}}})$ in the Gromov-Hausdorff distance as $t\rightarrow \infty$. 
\end{corollary}
\begin{proof}
This follows from Theorem \ref{thm_approx_order_iso}, Proposition \ref{prop_qgh_distance_appoi}, and Corollary \ref{cor_qgh_implies_gh}.  
\end{proof}


\section{Examples}\label{sec:examples}
\subsection{Classical compact groups acting on ${\rm C}^{*}$-algebras}\label{ss:recovering_classical_function_op_convolution} In this subsection, we investigate the implication of previous results when applied to the setting of classical compact groups acting on unital ${\rm C}^{*}$-algebras concretely represented over a compact homogeneous space, as mentioned in Remark \ref{rem:function_op_conv}. Suppose $G$ is a compact group, with associated Haar probability measure $\mu_{G}$. We let ${\rm Lip}(G)$ denote the space of complex-valued Lipschitz functions on $G$. Let $K\subset G$ be a subgroup of $G$ and let $X=G\setminus K$ denote our homogeneous space. We make the following assumptions on $X$ and $G$:
\begin{itemize}
    \item[(i)] The action $G\curvearrowright X$ is continuous and transitive. 
    \item[(ii)] $X$ supports a finite measure $\mu$ which is invariant under the (left) action of $G$, i.e. $\mu(gA)=\mu(A)$ for each $g\in G$ and $A\in \mathfrak{B}(X)$. 
    \item[(iii)] $G$ supports a $G$-invariant metric $d$, i.e. for each $g,h,\ell\in G$, $d(gh,g\ell)=d(h,\ell)$.  
\end{itemize}
At first glance, (ii) and (iii) may seem restrictive. Note, however, that if $X$ supports a finite measure, the measure $\hat{\mu}$ by 
\begin{align*}
    \hat{\mu}(A):=\int_{G}\mu(gA)d\mu_{G}(g)
\end{align*}
is invariant under the action of $G$. Likewise if $d$ is any metric on $G$, then it can be shown that the function $\hat{d}:G\times G\to \mathbb{R}^{+}$ by 
\begin{align*}
    \hat{d}(h,\ell)=\int_{G}d(g\ell,gh)d\mu_{G}(g)
\end{align*}
is a metric on $G$ which is invariant under the action of $G$. Now, assume $G$ acts on some unital ${\rm C}^{*}$-algebra $\mathcal{A} \subseteq \mathcal{B}(L^{2}(X, \mu))$. Thus, there exists a continuous unital group homomorphism $\sigma: G\rightarrow {\rm Aut}(\mathcal{A})$. 

If $\{\varphi_{t}\}_{t > 0} \subseteq L^{1}(G, \mu_{G})$ is an $L^{1}(\mathbb{G})$-approximate identity in the sense of Definition \ref{def:l_1_approx}, this means we have elements such that
\begin{itemize}
    \item[(i)] $\varphi_{t}(g) \geq 0$ for each $g \in G$ and $t > 0$;
    \item[(ii)] $\|\varphi_{t}\|_{L^{1}(G, \mu_{G})} = 1$ for all $t > 0$;
    \item[(iii)] For each $\delta > 0$, 
    \begin{gather*}
        \int\limits_{B(e, \delta)^{c}}\varphi_{t}(g)d\mu_{G}(g)\rightarrow 0,
    \end{gather*}
    \noindent as $t \rightarrow 0$. 
\end{itemize}

\begin{lemma}\label{approxid}
    Suppose $\{\varphi_{t}\}_{t>0}$ is an $L^{1}(G,\mu_{G})$ approximate identity as before. Then, for each $f\in C(G)$, $f*\varphi_{t}\to f$ uniformly, as $t\to 0$, where
    \begin{align*}
        f*\varphi_{t}(g)=\int_{G}f(h)\varphi_{t}(gh)d\mu_{G}(h).
    \end{align*}
\end{lemma}
\begin{proof}
    Let $f\in C(G)$. Since $G$ is compact, $f$ is uniformly continuous. Let $\varepsilon>0$. By the uniform continuity, there is an $\delta>0$ such that for any $g,h$ such that $d(g,h)<\delta$, $|f(g)-f(h)|<\varepsilon/2$. Likewise, by (iii), there is an $t>0$ sufficiently small such that 
    \begin{align*}
        \int_{B(e,\delta)^{c}}\varphi_{t}(h)d\mu_{G}(h)<\varepsilon/4\|f\|_{\infty}.
    \end{align*}
    Now, note that for any $g\in G$, using (i) and (ii) and the same $\delta$ as before we have 
    \begin{eqnarray*}
    |f\ast\varphi_{R}(g)-f(g)| 
    & = &
    \bigg|\int\limits_{G}f(h)\varphi_{t}(gh)d\mu_{G}(h)-\int\limits_{G}f(g)\varphi_{t}(gh)d\mu_{G}(h)\bigg| \\
    & = &
    \bigg|\int\limits_{G}(f(h)-f(g))\varphi_{t}(gh)d\mu_{G}(h)\bigg| \\
    & \leq &
    \int\limits_{B(g, \delta)}|f(h)-f(g)|\varphi_{t}(gh)d\mu_{G}(h)+\int\limits_{B(g, \delta)^{c}}|f(h)-f(g)|\varphi_{t}(gh)d\mu_{G}(h) \\
    & \leq &
    \epsilon/2+2\|f\|_{\infty}\int\limits_{B(g, \delta)^{c}}\varphi_{t}(gh)d\mu_{G}(h).
    \end{eqnarray*}
    Now, by the invariance of the metric and Haar measure under the group action, we have that 
    \begin{align*}
        \int_{B(g,\delta)^{c}}\varphi_{t}(gh)d\mu_{G}(h)=\int_{B(e,\delta)^{c}}\varphi_{t}(h)d\mu(h).
    \end{align*}
    Thus, 
    \begin{align*}
        \|f*\varphi_{t}-f\|_{\infty}\leq\varepsilon/2 +2\|f\|_{\infty}\int_{B(e,\delta)^{c}}\varphi_{t}(h)d\mu_{G}(h)<\varepsilon,
    \end{align*}
    for $t>0$ sufficiently small. 
\end{proof}
\begin{corollary}\label{dzero}
    Let $d(g)=d(g,e)$. Then, $d*\varphi_{t}(e)\to 0$ as $t\to\infty$. 
\end{corollary}
Before proceeding to the next example, we make the following observation concerning the set $K$, the seminorm $\|\cdot\|_{\ast}$ and the constants $\gamma_{t}$ coming from our general theorems. Namely, considering the Lipschitz algebra ${\rm Lip}(G)$ along with the role of the Lip-norm being played by the standard Lipschitz seminorm ${\rm Lip}$ here, we have that 
\begin{gather*}
    K = \{f \in C(G):\; {\rm Lip}(f) \leq 1, f(e) = 0\} \subseteq {\rm Lip}_{1}(G).
\end{gather*}
\noindent By duality and previous remarks, $\varepsilon_{G} \equiv \delta_{e}$ with $C(G)^{*} \cong M(G)$, and thus
\begin{gather*}
    \|\mu\|_{\ast} = \sup\{|\mu(f)|: \; f \in K\} = \sup\bigg\{\bigg|\int\limits_{G}f(g)d\mu(g)\bigg|: \; f \in K\bigg\}.
\end{gather*}
\noindent Therefore, 
\begin{align*}
    \gamma_{t}:=\|\eta_{t}\|_{*}=\|\varphi_{t}-\delta_{e}\|_{*}=&\sup_{f\in\text{Lip}_{1}(G)}\left|\int_{G}f(g)\varphi_{t}(g)d\mu_{G}(g)-f(e)\right|\\
    =&\sup_{f\in\text{Lip}_{1}(G)}\left|\int_{G}\left(f(g)-f(e)\right)\varphi_{t}(g)d\mu_{G}(g)\right|\\
    \leq& \; d*\varphi_{t}(e)\to 0,
\end{align*}
by Corollary \ref{dzero}. Note that Lemma \ref{approxid} may be viewed as an alternative proof of the results in Lemma \ref{lem:seminorm_convergence} for this specific setting.

Clearly, our assumptions are satisfied in the further specialized setting when $X=G$ for a compact abelian group. However, in this setting, we have the advantage of constructing $\varphi_{t}$ in a manner which makes the corresponding family finite rank operators when considered as operators on $L^{2}(G,\mu_{G})$, where $\mu_{G}$ is the corresponding (left and right invariant) normalized Haar measure. Going forward, we will restrict to the case where $\mathcal{A}=C(G)$, the continuous functions on $G$. We recall the following definition. 
    \begin{definition}
        Let $G$ be a compact Abelian group and let $\hat{G}$ denote the associated dual group of characters. We say that a family of sets $\{\Gamma_{N}\}_{N\geq 1}\subset\hat{G}$ is a F\o lner sequence if for each $\gamma\in \hat{G}$, 
        \begin{align*}
            \frac{\#\{(\gamma\cdot\Gamma_{N})\triangle \Gamma_{N}\}}{\#\Gamma_{N}}\to 0
        \end{align*}
        as $N\to\infty$, where $\triangle$ is the symmetric difference. 
    \end{definition}
     From this, we construct a family of kernels for which the corresponding kernel-integral operators on $L^{2}(G,\mu_{G})$ are finite-rank. Let $\{\Gamma_{N}\}_{N \geq 1}$ be a F\o lner sequence in $\hat{G}$. Consider the kernel 
    \begin{align*}
        k_{N}(g, h)=\frac{1}{\#\Gamma_{N}}\left|\sum_{\xi\in\Gamma_{N}}\xi(gh^{-1})\right|^{2}.
    \end{align*}
    In \cite{coifman}, it was shown that this kernel is an $L^{1}(\mathbb{G})$-approximate identity in the sense of Definition \ref{def:l_1_approx}. Clearly, by expansion of the modulus squared, integration against this kernel corresponds to finite-rank projections onto $\Gamma_{N}\cdot\Gamma_{N}$. Let $k_{N}(g)=k_{N}(g,e)$, where $e\in G$ is the group unit. Following notation established in previous sections, let $C(G)_{N}=P_{N}C(G)P_{N}=k_{N}\star C(G)$. We then obtain the following specialization of Corollary \ref{mainconv}.
    \begin{corollary}
        Let $C(G)$ be the continuous functions, and let $C(G)_{N}=k_{N}\star C(G)$. Then, as $N\to\infty$, 
        \begin{align*}
            d_{GH}\left(S\left(C(G)_{N}\right),S\left(C(G)\right)\right)\to 0.
        \end{align*}
    \end{corollary}
    We note here that in specializing to the case $G=\mathbb{T}^{d}$ for $1\leq d\leq \infty$, we obtain a generalization of the main result from \cite{LeimSuij24} to F\o lner sequences, as opposed to lattice points of $\mathbb{Z}^{d}$ contained in an expanding ball. 
\subsection{${\rm SU}_{q}(2)$ acting on the quantum sphere, and relations to Toeplitz quantization}
We first recall the definition of ${\rm SU}_{q}(2)$, its properties, and the corresponding quantum sphere as discussed in \cite{Podles87} (for greater detail on these points, see also \cite{Podles95, Timmermann08}). We fix $q \in \mathbb{R}$ such that $0 < |q| < 1$, and let $\mathcal{O}({\rm SU}_{q}(2))$ denote the universal unital $*$-algebra generated by $\alpha, \gamma$ subject to the relations 
\begin{gather*}
    \alpha\gamma = q\gamma\alpha, \;\;\;\; \alpha \gamma^{*} = q\gamma^{*}\alpha, \;\;\;\; \gamma\gamma^{*} = \gamma^{*}\gamma,
    \\ \alpha^{*}\alpha+\gamma^{*}\gamma = I, \;\;\;\; \alpha\alpha^{*}+q^{2}\gamma^{*}\gamma = I.
\end{gather*}
\noindent Define $\Delta: \mathcal{O}({\rm SU}_{q}(2))\rightarrow \mathcal{O}({\rm SU}_{q}(2))\otimes \mathcal{O}({\rm SU}_{q}(2))$ by its action on the generators
\begin{gather*}
    \Delta(\alpha) = \alpha\otimes \alpha-q\gamma^{*}\otimes \gamma,
    \\ \Delta(\gamma) = \gamma\otimes \alpha+\alpha^{*}\otimes \gamma,
\end{gather*}
\noindent and extending linearly; this endows $\mathcal{O}({\rm SU}_{q}(2))$ with the structure of a Hopf $*$-algebra. We let $C({\rm SU}_{q}(2))$ denote the ${\rm C}^{*}$-completion of $\mathcal{O}({\rm SU}_{q}(2))$; in this case, $\Delta$ extends to a well-defined co-multiplication on $C({\rm SU}_{q}(2))$, turning ${\rm SU}_{q}(2) = (C({\rm SU}_{q}(2)), \Delta)$ into a compact quantum group. Furthermore, that it is co-amenable follows from \cite[Theorem 2.12]{BedosMurphTus01}.

For the same parameter $q$, consider the universal unital $*$-algebra $P(S_{q\infty}^{2}) := \mathbb{C}\langle a, b\rangle/I_{q\infty}$, where $I_{q\infty}$ is the (two-sided) $*$-ideal inside the free $*$-algebra $\mathbb{C}\langle a, b\rangle$ generated by the relations
\begin{gather*}
    a^{*} = a, \;\;\;\; ba = q^{2}ab,
    \\ b^{*}b = -a^{2}+I,\;\;\;\; bb^{*} = -q^{4}a^{2}+I.
\end{gather*}
\noindent Then the \textit{quantum sphere} $S_{q\infty}^{2}$ is the ${\rm C}^{*}$-completion of $P(S_{q\infty}^{2})$, denoted $C(S_{q\infty}^{2})$.

We have an ergodic co-action of ${\rm SU}_{q}(2) \curvearrowright C(S_{q\infty}^{2})$ given by the unital $*$-homomorphism
\begin{gather*}
    \alpha_{q}: C(S_{q\infty}^{2})\rightarrow C(S_{q\infty}^{2})\otimes C({\rm SU}_{q}(2)), 
    \\ \alpha_{q}(a) := a \otimes 1_{{\rm SU}_{q}(2)}+c_{1}\otimes \gamma^{*}\gamma+b^{*}\otimes \alpha^{*}\gamma+b\otimes \gamma^{*}\alpha,
    \\ \alpha_{q}(b) := -qb^{*}\otimes \gamma^{2}+c_{1}\otimes \alpha\gamma+b\otimes \alpha^{2},
\end{gather*}
\noindent where $c_{1} := -(1+q^{2})a$, and then extending to all of $C(S_{q\infty}^{2})$. Thus, we are in the setting of Section \ref{sec:approximations_from_cqg_actions}. 

As discussed in the introduction, our approach in constructing approximating spaces for a target ${\rm C}^{*}$-algebra is motivated by kernel smoothing; such an idea also has significant connections to Toeplitz quantization. Our approach has potential use when the Toeplitz operators used in the quantization scheme have symbols originating in a noncommutative algebra, as compared to the classical setting. Such an approach has been taken up at various points in the recent past; see, for instance, \cite{Sontz13, Sontz14}, and \cite{Sontz16}. The latter is of particular interest, as a framework for constructing Toeplitz operators whose symbols come from ${\rm SU}_{q}(2)$ is investigated.

\subsection*{Funding}
D.~G.\ acknowledges support from the U.S.\ Department of Defense, Basic Research Offce, under Vannevar Bush Faculty Fellowship grant N00014-21-1-2946, managed by the Office of Naval Research, and from the U.S.\ Department of Energy under grant DE-SC0025101. T.~C. \ and G.~H.\ were supported as Postdoctoral Fellows from the first grant. 

\subsection*{Acknowledgments}
The authors would like to thank Malte Leimbach and Vishwa Dewage for useful discussions. Part of this work was supported by the Swedish Research Council under grant no. 2021-06594 while the authors D.~G. and G.~H.  were in residence at Institut Mittag-Leffler in Djursholm, Sweden, during spring 2026. The authors would like to thank the institute for their hospitality.



\begin{thebibliography}{99}

\bibitem{AgKaadKyed22}
{\sc K. \, Aguilar, J. \, Kaad, and D. \, Kyed},
{\it The  Podleś spheres converge to the sphere},
{\rm Comm. Math. Phys. 392 (2022), no. 3, 1029-1061}.

\bibitem{BedosMurphTus01}
{\sc E. \, Bédos, G. \, J. \, Murphy, and L. \, Tuset},
{\it Co-amenability of compact quantum groups},
{\rm J. Geom. Phys. 40 (2001), 130-153.}

\bibitem{bs}
{\sc A. \, Böttcher and B. \, Silbermann},
{\it Analyis of Toeplitz Operators},
{\rm Springer Monographs in Mathematics, Springer Berlin, Heidelberg, 2006.}





\bibitem{coifman}
{\sc R. \, Coifman and G. \, Weiss},
{\it Operators Associated with Representations of Amenable Groups Singular Integrals induced by Ergodic Flows, the Rotation Method and Multipliers. },
{\rm Studia Math. 47 (1973), 285-303.}

\bibitem{Connes89}
{\sc A. \, Connes},
{\it Compact metric spaces, Fredholm modules, and hyperfiniteness},
{\rm Ergodic Theory Dynam. Systems 9 (1989), 207-220.}

\bibitem{ConnesSuij21}
{\sc A. \, Connes, and W. \, D. \, van Suijlekom},
{\it Spectral truncations in noncommutative geometry and operator systems},
{\rm Comm. Math. Phys. 383 (2021), 2021-2067.}

\bibitem{dand_lizzi_mart}
{\sc F. \, D'Andrea, F. \, Lizzi, and P. \, Martinetti},
{\it Spectral geometry with a cut-off: topological and metric aspects},
{\rm J. Geom. Phys. 82 (2014), 18-45.}

\bibitem{deCommer17}
{\sc K. \, De Commer},
{\it Actions of compact quantum groups},
{\rm Topological quantum groups, Banach Center Publ. Vol. 111, Polish Acad. Sci. Inst. Math., Warsaw, 2017.}


\bibitem{Evans22}
{\sc L. \, C. \, Evans},
{\it Partial Differential Equations},
{\rm American Mathematical Society, Rhode Island U.S.A, 2022.}



\bibitem{ExelNg02}
{\sc R. \, Exel, and C. \, K. \, Ng},
{\it Approximation property of ${\rm C}^{*}$-algebraic bundles},
{\rm Math. Proc. Cambridge Philos. Soc. 132 (2002), no. 3, 509-522.}

\bibitem{FulscheGalke25},
{\sc R. \, Fulsche, and N. \, Galke},
{\it Quantum harmonic analysis on locally compact abelian groups},
{\rm J. Fourier Anal. and Appl. 31 (2025), no. 1, 13pp.}

\bibitem{ges}
{\sc Y. \, Gaudillot-Estrada, and W. \, D. \, Suijlekom},
{\it Convergence of spectral truncations for compact metric groups},
{\rm Int. Math. Res. Not. 13 (2025), 1-9.}

\bibitem{HajMattSzy03}
{\sc P. \, M. \, Hajac, R. \, Matthes, and W. \, Szymanski},
{\it Quantum real projective spacees, disc and spheres},
{\rm Algebras and Rep. Theory 6 (2003), 169-192.}

\bibitem{HuNeufangRuan10}
{\sc Z. \, Hu, M. \, Neufang, and Z. \, J. \, Ruan},
{\it Multipliers on a new class of Banach algebras, locally compact quantum groups, and topological centres},
{\rm Proc. London Math. Soc. 100:(2) (2010), pp.429-458.}


\bibitem{kk}
{\sc J. \, Kaad, and D. \, Kyed},
{\it The quantum metric structure of quantum SU(2)},
{\rm Memoirs of the European Mathematical Society 18, EMS Press, Berlin, 2025.}

\bibitem{kk_two}
{\sc J. \, Kaad, and D. \, Kyed},
{\it A comparison of two quantum distances},
{\rm Math. Phys., Anal. and Geom. 26:(1), (2023), art. 8}.

\bibitem{Kerr03}
{\sc D. \, Kerr},
{\it Matricial quantum Gromov-Hausdorff distance},
{\rm J. Funct. Anal 205 (2003), iss. 1, 132-167.}

\bibitem{KerrLi09}
{\sc D. \, Kerr, and H. \, Li},
{\it On Gromov-Hausdorff convergence for operator  metric spaces},
{\rm J. Op. Theory 62 (2009), no. 1, pp.83-109.}

\bibitem{Lat16}
{\sc F. \, Latrémolière},
{\it The quantum Gromov-Hausdorff propinquity},
{\rm Trans. Amer. Math. Soc. 368 (2016), no. 1, 365-411.}

\bibitem{Lat22}
{\sc F. \, Latrémolière},
{\it The Gromov-Hausdorff propinquity for metric spectral triples},
{\rm Adv. Math. 404 (2022), no. 108393, 56.}

\bibitem{Leimbach25}
{\sc M. \, Leimbach},
{\it Convergence of Peter-Weyl truncations of compact quantum groups},
{\rm J. Noncommut. Geom. (2025), published online first, DOI 10.4171/JNCG/634.}

\bibitem{LeimSuij24}
{\sc M. \, Leimbach, and W. \, D. \, van Suijlekom},
{\it Gromov-Hausdorff convergence of spectral truncations for tori},
{\rm Adv. Math. 439 (2024) 109496}.

\bibitem{Li03}
{\sc H. \, Li},
{\it ${\rm C}^{*}$-algebraic quantum Gromov-Hausdorff distance},
{\rm preprint (2003), arXiv:math/0312003.}

\bibitem{Li06}
{\sc H. \, Li},
{\it Order-unit quantum Gromov-Hausdorff distance},
{\rm J. Funct. Anal. 231 (2006), no. 2, 312-360. MR 2195335}

\bibitem{Li09}
{\sc H. \, Li},
{\it Compact quantum metric spaces and ergodic actions of compact quantum groups},
{\rm J. Funct. Anal. 256 (2009), no. 10, 3368-3408. MR 2504529}

\bibitem{OzawaRieffel05}
{\sc N. \, Ozawa, and M. \, A. \, Rieffel},
{\it Hyperbolic group ${\rm C}^{*}$-algebras and free-product ${\rm C}^{*}$-algebras as compact quantum metric spaces},
{\rm Canadian J. Math. 57 (2005), no. 5, 1056-1079.}

\bibitem{paulsen}
{\sc V. \, I. \, Paulsen},
{\it Completely Bounded Maps and Operator Algebras},
{\rm Cambridge University Press, Cambridge, 2002}.


\bibitem{Podles87}
{\sc P. \, Podleś},
{\it Quantum spheres},
{\rm Lett. Math. Phys. 14 (1987), 193-202.}

\bibitem{Podles95}
{\sc P. \, Podleś},
{\it Symmetries of quantum spaces. Subgroups and quotient spaces of quantum ${\rm SU}(2)$ and ${\rm SO}(3)$ groups.}
{\rm Comm. Math. Phys. 170 (1995), pp.1-20.}

\bibitem{rieffel}
{\sc M. \, Rieffel},
{\it Metrics on states from actions of compact groups},
{\rm Doc. Math. 3 (1998), 215-229}.

\bibitem{rieffel_two}
{\sc M. \, Rieffel},
{\it Gromov-Hausdorff distance for quantum metric spaces},
{\rm Mem. Amer. Math. Soc. 168 (2004), no. 768, 73 pp}.

\bibitem{rieffel_three}
{\sc M. \, Rieffel},
{\it Convergence of Fourier truncations for compact quantum groups and finitely generated groups},
{\rm J. Geom. Phys. 192 (2023), 104921.}


\bibitem{Sontz13}
{\sc S. \, B. \, Sontz},
{\it A Reproducing Kernel and Toeplitz Operators in the Quantum Plane},
{\rm Comm. in Math. 21 (2013), 137-160.}

\bibitem{Sontz14}
{\sc S. \, B. \, Sontz},
{\it Paragrassmann Algebras as Quantum Spaces, Part II: Toeplitz Operators},
{\rm J. Operator Theory 71 (2014), 411-426.}

\bibitem{Sontz16}
{\sc S. \, B. \, Sontz},
{\it Toeplitz Quantization for Non-commuting Symbol Spaces such as ${\rm SU}_{q}(2)$},
{\rm Comm. in Math. 24 (2016), iss. 1.}

\bibitem{Suij21}
{\sc W. \, D. \, van Suijlekom},
{\it Gromov-Hausdorff convergence of state spaces for spectral truncations},
{\rm J. Geom. Phys. 162 (2021), 104075}.

\bibitem{Takesaki03}
{\sc M. \, Takesaki},
{\it Theory of Operator Algebras I},
{\rm Springer-Verlag, New York, 2003.}

\bibitem{Timmermann08}
{\sc T.\ Timmermann},
{\it An Invitation to Quantum Groups and Duality: From Hopf Algebras to Multiplicative Unitaries and Beyond},
{\rm EMS Textbooks in Mathematics 5, EMS Press, 2008}.

\bibitem{Toyota23}
{\sc R. \, Toyota},
{\it Quantum Gromov-Hausdorff convergence of spectral truncations for groups with polynomial growth},
{\rm preprint, (2023), arXiv:2309.13469.}

\bibitem{Werner84}
{\sc R. \, Werner},
{\it Quantum harmonic analysis on phase space},
{\rm J. Math. Phys. 25 (1984), no. 5, 1404-1411.}

\end{thebibliography}
\end{document}